\documentclass[11pt]{article} 
\usepackage{amssymb}
\usepackage{amsmath}
\usepackage{CJKutf8} 
\usepackage{geometry}
\usepackage{braket}
\usepackage{url}
\usepackage{algorithm}
\usepackage{algpseudocode}
\usepackage{setspace}
\usepackage{hyperref}

\usepackage{caption}
\usepackage{subcaption}
\usepackage[shortlabels]{enumitem}

\DeclareMathOperator*{\rank}{rank}
\DeclareMathOperator*{\diag}{diag}
 \usepackage{xcolor}
\usepackage{graphicx} 
\usepackage{multirow}
\usepackage{amsthm}
\usepackage{graphicx}
\newtheorem{theorem}{Theorem}
\newtheorem{lemma}[theorem]{Lemma}

\newtheorem{remark}{Remark}
\newcommand{\R}{\mathbb R}
\renewcommand{\algorithmicrequire}{\textbf{Input: }}
\renewcommand{\algorithmicensure}{\textbf{Output: }}

\author{Daniel Kressner\thanks{Institute of Mathematics, EPF Lausanne, 1015
Lausanne, Switzerland. E-mails: \texttt{daniel.kressner@epfl.ch}, \texttt{\# hysan.lam@epfl.ch}. The work of both authors was supported by the SNSF research project \textit{Fast algorithms from low-rank updates}, grant number: 200020\_178806.}\and Hei Yin Lam\footnotemark[1]}
\date{\today}

\title{Randomized low-rank approximation \\ of parameter-dependent matrices}

\begin{document}
\maketitle
\begin{abstract}
This work considers the low-rank approximation of a matrix $A(t)$ depending on a parameter $t$ in a compact set $D \subset \R^d$. Application areas that give rise to such problems include computational statistics and dynamical systems. Randomized algorithms are an increasingly popular approach for performing low-rank approximation and they usually proceed by multiplying the matrix with random dimension reduction matrices (DRMs). Applying such algorithms directly to $A(t)$ would involve different, independent DRMs for every $t$, which is not only expensive but also leads to inherently non-smooth approximations. In this work, we propose to use constant DRMs, that is, $A(t)$ is multiplied with the same DRM for every $t$. The resulting parameter-dependent extensions of two popular randomized algorithms, the randomized singular value decomposition and the generalized Nystr\"{o}m method, are computationally attractive, especially when $A(t)$ admits an affine linear decomposition with respect to $t$. {We perform a probabilistic analysis for both algorithms, deriving bounds on the expected value as well as failure probabilities for the approximation error when using Gaussian random DRMs}. Both, the theoretical results and numerical experiments, show that the use of constant DRMs does not impair their effectiveness; our methods reliably return quasi-best low-rank approximations.
\end{abstract}

\section{Introduction}

 This work is concerned with performing the low-rank approximation of a parameter-dependent matrix $A(t)\in \mathbb{R}^{m\times n}$, $m\geq n$ for many values of the parameter $t\in D \subset \R^d$. This task arises in a variety of application areas, including Gaussian process regression~\cite{kressner2020certified} and time-dependent data from dynamical systems, image processing tasks or natural language processing~\cite{nonnenmacher2008dynamical}; see Section~\ref{sect: numerical} for more concrete examples.
 
When $D$ is an interval, say $D = [0,T]$, a low-rank approximation of $A(t)$ can be obtained by considering $t$ as time and solving suitable differential equations. For example, under certain smoothness assumptions on $A(t)$, a smooth variant of the singular value decomposition (SVD) of $A(t)$ satisfies a system of differential equations~\cite{baumann2003singular,bunse1991numerical,dieci1999smooth}. For the purpose of low-rank approximation, a cheaper and simpler approach follows from the fact that the set of fixed-rank matrices forms a smooth manifold.
 In turn, differential equations for the low-rank factors can be obtained from projecting $\dot{A}(t)$ on the tangent space of this manifold, leading to the so-called dynamical low-rank approximation~\cite{koch2007dynamical,nonnenmacher2008dynamical}.  Due to the stiffness introduced by the curvature of the manifold, the numerical implementation of this approach comes with a number of challenges that have been addressed during the last years \cite{ceruti2022unconventional, Lubich2014projector}.  Reduced basis methods \cite{Hesthaven2016Certified,Quarteroni2016Reduced} provide a general framework for solving large-scale parameter-dependent problems. In~\cite{kressner2020certified}, this framework has been adapted to yield a parametric variant of adaptive cross approximation for symmetric positive definite $A(t)$, which
 constructs a low-rank approximation by sampling entries of $A(t)$.
 
When $A(t) \equiv B$ for a constant matrix $B\in \mathbb{R}^{m\times n}$, $m\geq n$, it is well known that the truncated SVD achieves the best approximation of fixed rank $r$~\cite{horn2012matrix}. However, the computational cost~\cite[section 8.6.3]{golub2013matrix} of the SVD renders it infeasible for large matrices. The randomized SVD, referred as the Halko-Martinson-Tropp (HMT) method \cite{halko2011finding} in the following, is an increasingly popular alternative that attains near-optimal approximation quality at significantly lower cost. HMT first creates a sketch of the matrix $B$ by multiplying it with a random dimension reduction matrix (DRM) $\Omega\in \mathbb{R}^{n\times (r+p)}$, with some small integer $p$, and then obtains a low-rank approximation by orthogonally projecting $B$ onto the range of $B\Omega$. While the HMT method is easy to implement and effective, it requires two passes over $B$ and it is thus not streamable in the sense of~\cite{clarkson2009numerical}. Alternatively, the generalized Nystr\"{o}m method \cite{clarkson2009numerical,tropp2017practical} utilizes two sketches: $B\Omega$ and $\Psi^TB$ with random DRMs $\Omega\in \mathbb{R}^{n\times (r+p)}$, $\Psi\in \mathbb{R}^{m\times (r+p+\ell)}$ for some additional small integer $\ell>0$. A low-rank approximation is obtained by oblique projection,
which makes the generalized Nystr\"{o}m method less accurate compared to the HMT method but it still achieves near-optimal approximation quality~\cite{nakatsukasa2020fast, tropp2017practical}. Notably, the generalized Nystr\"{o}m method is streamable and requires only one pass
over the matrix $B$.

In this work, we aim at extending randomized low-rank approximation to a parameter-dependent matrix $A(t)\in \mathbb{R}^{m\times n}$. The most straightforward extension certainly consists of applying HMT or generalized Nystr\"{o}m from scratch for every $t$. In the case of HMT, this corresponds to using a DRM $\Omega(t)$ that is different and, in fact, independent for every $t$. {Clearly, this does not allow one to exploit structure, such as smoothness properties of $A(t)$. For example, when the range of $A(t)$ changes very little when moving from one parameter value to the next, it appears to be wasteful to recompute the full sketch.} Moreover, the resulting approximation does not inherit any smoothness from $A$ due to the non-smooth of the DRM.
{We avoid these drawbacks by using \emph{constant} DRMs, that is DRMs which do not depend on $t$}. {This not only avoids most of the cost for generating DRMs but it can also reduce significantly the cost of constructing the sketches and, in turn, the low-rank approximation itself. For example, consider an affine linear dependence commonly found in the literature \cite{kressner2020certified,rozza2008reduced,sirkovic2016subspace}: $A(t)=\sum^k_{i=1}\varphi_i(t)A_i$ for scalar functions $\varphi_i:D \rightarrow \mathbb{R}$ and constant matrices $A_i\in\mathbb{R}^{m\times n}$. The sketches $A_i\Omega$ are computed a priori, in an offline phase. In the subsequent online phase, the evaluation of a sketch $A(t)\Omega=\sum^k_{i=1}\varphi_i(t)A_i\Omega$ only requires evaluating scalar functions and adding matrices together.} {Perhaps surprisingly, our analysis and numerical experiments reveal that using constant instead of independent Gaussian random DRMs for HMT and generalized Nystr\"{o}m only has a minor impact on the $L^2(D)$  approximation error, provided that $A(t)$ is continuous with respect to $t$. While it may be more natural to consider the worst-case error with respect to $t$ in applications, the analysis is significantly more complicated. As a first step in this direction, we provide such an analysis for the HMT method when $A(t)$ is Lipschitz continuous and $D=[0,T]$.} 
%both methods are expected to attain near-optimal $L^2$ approximation error when using Gaussian random matrix as DRM. To further quantify the error, we also provide a tail probability error bound for the HMT method for parameter-dependent matrices. 

The remainder of the paper is organized as follows. In Section 2, after recalling existing results for HMT, we present our HMT method for parameter-dependent matrices and provide bounds on the expected error as well as the failure probability that the error is significantly larger than the best $L^2$ approximation error. {To further quantify the error, we also provide a first failure probability bound in the uniform norm. %These generalize results from~\cite{halko2011finding}.
} In Section 3, after recalling the generalized Nystr\"{o}m method, we extend it to parameter-dependent matrices and provide a bound on the expected error that generalizes results from~\cite{nakatsukasa2020fast, tropp2017practical}. We also provide a bound on the failure probability, which even in the constant case is new. In Section 4, we report numerical experiments to illustrate our theoretical findings.

\section{The HMT method for parameter-dependent matrices}
 Before we describe and analyze its extension to parameter-dependent
matrices, we briefly review the theory of the HMT method from~\cite{halko2011finding} for a constant matrix $B\in \mathbb{R}^{m\times n}$.

\subsection{HMT for a constant matrix}

Given a random DRM $\Omega \in \mathbb{R}^{n\times (r+p)}$, the HMT method constructs a low-rank approximation by orthogonally projecting $B\in \mathbb{R}^{m\times n}$ onto the column space of $B\Omega$, i.e.,
\begin{equation*}
    B\approx \mathcal{P}_{B\Omega} B,
\end{equation*}
where ${P}_{B\Omega}:=(B\Omega)(B\Omega)^\dagger$ and $\dagger$ denotes the Moore-Penrose pseudoinverse.
To compare the approximation error of HMT with the best rank-$r$ approximation error of $B$, let us consider the SVD of $B=U\Sigma V^T$, partitioned as follows:
\begin{equation}
\label{partition}
    B=U\begin{bmatrix}
    \Sigma_1& 0 \\
    0 &\Sigma_2  \\
   
    \end{bmatrix}V^T=U
    \begin{bmatrix}
    \Sigma_1& 0 \\
    0 &\Sigma_2 \\

    \end{bmatrix}\begin{bmatrix}
    V_1^T \\
    V_2^T \\
 
    \end{bmatrix},
\end{equation} where $\Sigma_1\in \mathbb{R}^{r\times r}$, $\Sigma_2\in \mathbb{R}^{(m-r)\times (n-r)}$, $V_1 \in \mathbb{R}^{n\times r}$, and $V_2 \in \mathbb{R}^{n\times (n-r)}$. The probabilistic error analysis of HMT relies on the following structural bound.
\begin{theorem}[{\cite[Theorem 9.1]{halko2011finding}}]
\label{deter bound}
With the notation introduced above, assume that $V_1^T\Omega$ has full row rank. Then
$
    \|(I-\mathcal{P}_{B\Omega})B\|_F^2\leq \|\Sigma_2\|_F^2+\|\Sigma_2(V_2^T\Omega)(V_1^T\Omega)^\dagger\|_F^2.
$
\end{theorem}
For specific random DRMs, the result of Theorem~\ref{deter bound} can be turned into bounds on the expected error and failure probability. In particular, this is the case when $\Omega$ is a Gaussian random matrix, i.e., the entries are i.i.d. standard normal random variables.
\begin{theorem}[{\cite[Theorem 10.5]{halko2011finding}}]
\label{HMT}
Let $\sigma_1(B)\geq \sigma_2(B)\geq \cdots$ denote the singular values of $B\in \R^{m\times n}$.
Suppose that $\Omega$ is an $n\times (r+p)$ Gaussian random matrix, where $r\geq 2$, $p \geq 2$, and $r+ p \leq \min\{m, n\}$. 
Then 
\begin{equation*}
    \mathbb{E}\|(I-\mathcal{P}_{B\Omega})B\|^2_F\leq \Big(1+\frac{r}{p-1}\Big)\Big(\sum_{j>r}\sigma_j(B)^2\Big).
\end{equation*}
\end{theorem}
\begin{theorem}[{\cite[Theorem 10.7]{halko2011finding}}]
\label{HMT tail} Suppose that the assumptions of Theorem~\ref{HMT} and, additionally, $p\ge 4$ hold.
Then for all $u,\gamma\geq 1$, the probability that the inequality
\begin{equation} \label{eq:hmttail}
\|(I-\mathcal{P}_{B\Omega})B\|_F\leq \Big(1+\gamma\cdot \sqrt{\frac{3r}{p+1}}\Big)\Big(\sum_{j>r}\sigma_j(B)^2\Big)^{\frac{1}{2}}+u\gamma\cdot\frac{e\sqrt{r+p}}{p+1}\cdot \sigma_{r+1}(B),
\end{equation} fails is at most $2\gamma^{-p}+e^{-u^2/2}$.\end{theorem}
%Instead of standard Gaussian DRM, subsampled random Fourier transform is also used as DRM and the error analysis is also contained in \cite{halko2011finding}. 

When implementing HMT, one usually computes the economy-size QR factorization $B\Omega = QR$, that is, $Q$ has orthonormal columns and $R$ is a square upper triangular matrix. One then obtains $\mathcal{P}_{B\Omega} B$ from $Q ( Q^TB )$. The computation of the sketch $B\Omega$ typically dominates the overall computational effort. For example, when $B,\Omega$ are dense unstructured matrices then the sketch requires $O(mn(r+p))$ operations while the overall cost of HMT is $O(mn(r+p)+m(r+p)^2)$.

\subsection{HMT for a parameter-dependent matrix}

We now extend HMT to a parameter-dependent matrix $A \in {C}(D,\mathbb{R}^{m\times n})$, that is, the entries of $A(t)$ depend continuously on a parameter $t \in D \subset \R^d$.
We suggest to compute an approximation of the form 
\begin{equation} \label{eq:parahmt}
    A(t)\approx \mathcal{P}_{A(t)\Omega}A(t)=\big(A(t)\Omega\big)\big(A(t)\Omega\big)^\dagger A(t) \;\;\;\text{for all }t\in D,
\end{equation} for a \emph{constant} (in $t$) DRM $\Omega\in \mathbb{R}^{n\times (r+p)}$.  
For general $A(t)$, to form the approximation \eqref{eq:parahmt}, we need to obtain $A(t)\Omega$, which appears to offer no advantage in terms of computational cost. However, the use of a constant DRM comes with the potential advantage of saving computations when the approximation is performed repeatedly for different $t$ in some cases. This becomes particularly evident when $A(t)\in\mathbb{R}^{m\times n}$ admits an affine linear decomposition of the form \begin{equation}
\label{affine combination}
    A(t)=\sum^k_{i=1}\varphi_i(t)A_i,
\end{equation} for continuous $\varphi_i:D\to \R$ and constant $A_i\in \mathbb{R}^{m\times n}$. 

%\textcolor{blue}{Suppose an approximation or the exact $\varphi_i$ and $A_i$ are known.}
When $k$ remains modest and the approximation~\eqref{eq:parahmt} is evaluated for many different $t$, the overall cost can be reduced significantly by first computing and storing the sketches $X_i = A_i\Omega$ for $i=1,\ldots,k$ during an offline phase. Subsequently, during an online phase, the sketch $A(t) \Omega = \sum^k_{i=1}\varphi_i(t)X_i$ is cheaply obtained by a linear combination of the matrices $X_i$. To reduce the cost of the next step of HMT, which requires a QR decomposition of $A(t)\Omega$, we precompute an economy-size QR decomposition
\begin{equation} \label{eq:qr}
\begin{bmatrix}
X_1 &X_2 &\cdots &X_k
\end{bmatrix} = QR
\end{equation}
as well as the sketches $Y_i = Q^T X_i$ and $Z_i = A_i^T Q$ for $i = 1,\ldots,k$. As the range of $A(t)\Omega$ is contained in the range of $Q$, {we have
\[
A(t)\Omega = Q Q^T A(t)\Omega =Q\Big(\sum_{i=1}^{k}\varphi_i(t)Q^TA_i\Omega\Big)= Q\Big(\sum_{i=1}^{k}\varphi_i(t) Y_i\Big).
\]
We compute $\tilde Q_t R_t =  \sum_{i=1}^{k}\varphi_i(t) Y_i$ the economy-size QR decomposition of a $k(r+p)\times (r+p)$ matrix and let $Q_t=Q\tilde Q_t$. After that, we obtained an economy-size QR decomposition \[
Q_tR_t= Q\tilde Q_t R_t =A(t)\Omega .
\]} This 
 yields the low-rank approximation \begin{equation*}
    A(t)\approx Q_t Q_t^TA(t) =Q_t \Big[\Tilde{Q}_t^T\Big(\sum^k_{i=1}\varphi_i(t)Q^TA_i\Big)\Big] = 
   Q_t \Big[\Big(\sum^k_{i=1}\varphi_i(t) Z_i\Big) \Tilde{Q}_t \Big]^T.
\end{equation*}
Algorithms \ref{offline phase 2} and \ref{online phase 2} summarize the described procedure. 
\begin{algorithm}[ht]
\caption{HMT for $A(t)=\sum^k_{i=1}\varphi_i(t)A_i$ (offline phase)}\label{offline phase 2}
\algorithmicrequire Matrices $A_i\in \mathbb{R}^{m \times n}$ for $i=1,\ldots,k$ and integer $r+p$.\\
\algorithmicensure  Matrices $Q \in \R^{m\times k(r+p)}$, $Y_i \in \R^{k(r+p)\times (r+p)}$ and $Z_i \in \R^{n\times k(r+p)}$ for $i=1,\ldots,k$.
     \begin{algorithmic}[1]
     \State Generate DRM $\Omega\in \mathbb{R}^{n\times (r+p)}$.
      \State Compute $X_1 = A_1\Omega,\ldots ,X_k = A_k\Omega$.
      \State Compute economy-size QR factorization $\begin{bmatrix}
X_1 &X_2 &\cdots &X_k
\end{bmatrix}=QR.$
\State Compute $Y_1 = Q^TX_1,\ldots ,Y_k = Q^TX_k$.
\State Compute $Z_1 = A^T_1 Q,\ldots ,Z_k = A_k^T Q$.
     \end{algorithmic}
\end{algorithm}
\begin{algorithm}[ht]
\caption{HMT for $A(t)=\sum^k_{i=1}\varphi_i(t)A_i$ (online phase)}\label{online phase 2}
\algorithmicrequire Matrices $Q$, $Y_i$, $Z_i$, functions $\varphi_i \in C(D,\mathbb R)$, for $i=1,\ldots, k$, evaluation points $t_1,\ldots,t_q \in D$.\\
\algorithmicensure  Matrices {$Q_{t_j}, W_{t_j}$ defining low-rank approximations $A(t_j) \approx Q_{t_j} W_{t_j}^T$ for $j=1,\ldots,q$.}
     \begin{algorithmic}[1]
     \For{$j=1,\ldots,q$}
      
      \State Compute economy-size QR factorization $\sum^k_{i=1}\varphi_i(t_j)Y_i=\Tilde{Q}_{t_j} {R_{t_j}}$.
      \State Compute $W_{t_j} = \big(\sum^k_{i=1}\varphi_i(t_j)Z_i\big) \Tilde{Q}_{t_j}$.
      \State Compute $Q_{t_j}=Q\Tilde{Q}_{t_j}$.
     \EndFor
     \end{algorithmic}
\end{algorithm}

{We now discuss the computational cost of Algorithm \ref{offline phase 2} and Algorithm \ref{online phase 2}. We assume that $p = O(r)$ and we let $c_A$ denote an upper bound on the cost of multiplying $A(t)$ or any of the matrices $A_i$ (or their transposes) with a vector. For example, when $A(t)$ is dense and unstructured matrix, then $c_A=2mn$. Using that the economy-size QR factorization in Algorithm~\ref{offline phase 2} requires $O(mk^2 r^2)$ operations, the total cost of the offline phase is
\begin{equation} \label{eq:offline}
O( k^2 r c_A + k^2 r^2 m).
\end{equation}
On the other hand, the online phase requires
\begin{equation} \label{eq:online}
O\big( q\big(k^2 r n+k r^2 (n+m)\big) \big)
\end{equation}
operations. In contrast, applying the standard HMT method $q$ times to each matrix $A(t_j)$ requires
$
O\big( q ( r c_A + r^2 m) \big)
$
operations. Comparing with~\eqref{eq:offline} and~\eqref{eq:online}, the cost of Algorithms~\ref{offline phase 2} and~\ref{online phase 2} compares favorably when $k^2\ll q$ and $\max\{k,r\} \cdot k \ll c_A / m$ or, in other words, when $k$ remains small and it is not cheap to apply matrix-vector product with $A(t)$ or $A(t)^T$. This discussion implicitly assumes that the same value of $r+p$ leads to a comparable approximation error in both approaches, which is indeed confirmed by the numerical  experiments (see Section \ref{sect: numerical}) as well as the theoretical results in the next section.}
\subsection{Error analysis}
\subsubsection{$L^2$ error}
Throughout this section, we will assume that $D\subset \mathbb{R}^d$ is compact. To perform a probabilistic error analysis of the approximation~\eqref{eq:parahmt} for a Gaussian random matrix $\Omega$, 
we first show that the approximation error is measurable.
\begin{lemma}
\label{measurable}
Let $\Omega$ be an $n\times (r+p)$ Gaussian random matrix. Consider the measure space $(D,\mathcal{B}(D),\lambda)$ and the probability space $(\mathbb{R}^{n\times (r+p)},\mathcal{B}({\mathbb{R}^{n\times (r+p)}}),\mu_{\Omega})$, where $\mathcal{B}(\cdot)$ denotes the Borel $\sigma$-algebra of a set, $\lambda$ is the Lebesgue measure and $\mu_\Omega$ is the distribution of $\Omega$. If $A\in {C}(D,\mathbb{R}^{m\times n})$ then the function $f:D\times\mathbb{R}^{n\times (r+p)}\rightarrow \mathbb{R}$ defined by
\begin{equation*}
    f(t,Y):= \|(I-\mathcal{P}_{A(t)Y})A(t)\|^2_F,
\end{equation*} is measurable on the product measure.
\end{lemma}
\begin{proof}
 For $k \ge 1$, let us define the function
\begin{equation*}
     f_k(t,Y):=\Big\|\Big[I-\big(A(t)Y\big)\big(Y^TA(t)^TA(t)Y+\frac{1}{k}I\big)^{-1}\big(A(t)Y\big)^T\Big]A(t)\Big\|^2_F,
\end{equation*}
which is continuous with respect to $t$ and $Y$. In turn, each $f_k:D\times \mathbb{R}^{n\times(r+p)}\rightarrow \mathbb{R}$ is a Carath\'eodory function~\cite[Definition 4.50]{guide2006infinite}. Additionally using that $\mathbb{R}^d$, $D$ are metrizable and separable, this allows us 
%\textcolor{blue}{Note that $\mathbb{R}^d$ with the usual topology is metrizable \cite[Theorem 20.3]{Munkres2000Topology} and second countable \cite[P.190]{Munkres2000Topology}, then the subspace $D$ is also metrizable and second countable. Using \cite[Lemma 3.4]{guide2006infinite}, the subspace $D$ is a separable metrizable space}. By
to apply~\cite[Lemma 4.51]{guide2006infinite} and conclude that $f_k$ is measurable on the product measure.
By \cite[Theorem 4.3]{barata2012moore}, \[ \big(Y^TA(t)^TA(t)Y+\frac{1}{k}I\big)^{-1}\big(A(t)Y\big)^T \stackrel{k\to\infty}{\to} \big(A(t)Y\big)^\dagger \] and thus 
$f_k(t,Y) \stackrel{k\to\infty}{\to} f(t,Y)$ pointwise. 
By \cite[Corollary 8.9]{schilling2017measures}, the pointwise limit of a sequence of measurable real-valued functions is measurable; therefore $f(t,Y)$ is also measurable on the product measure.
\end{proof}

\begin{remark}
Note that the function $f$ defined in Lemma~\ref{measurable} can be discontinuous. For example, let $t\in [-1,1]$ and
\begin{equation*}
A(t)=\begin{bmatrix}
    1 & 0 & 0 \\
    0 & t & 1 \\
    0 & 0 & 0 \\
\end{bmatrix}, \quad  Y=\begin{bmatrix}
1 & 0 \\
0 & 1 \\
0 & 0
\end{bmatrix}.
\end{equation*} While 
$
    f(t,Y)=\|(I-\mathcal{P}_{A(t)Y})A(t)\|_F^2=0
$ for every $t\neq 0$, we have
\begin{equation*}
    f(0,Y)=\|(I-\mathcal{P}_{A(0)Y})A(t)\|_F^2=\left\|
    \begin{bmatrix}
    1 & 0 & 0 \\
    0 & 0 & 1 \\
    0 & 0 & 0 \\
\end{bmatrix} - \begin{bmatrix}
    1 & 0 & 0 \\
    0 & 0 & 0 \\
    0 & 0 & 0 \\
\end{bmatrix}\right\|_F^2=1.
\end{equation*}
\end{remark}

Lemma~\ref{measurable} allows us to apply Tonelli's theorem in order to bound the expected error of our approximation~\eqref{eq:parahmt}.
\begin{theorem}
\label{l2-exp}
Consider $A\in {C}(D,\mathbb{R}^{m\times n})$ and an $n\times (r+p)$ Gaussian random matrix  $\Omega$ with $r\ge 2$, $p \ge 2$. Then 
\begin{equation} \label{hmtexp-bound}
    \mathbb{E}\Big[\int_{D}\|(I-\mathcal{P}_{A(t)\Omega})A(t)\|^2_F\,\mathrm{d}t\Big]\leq \Big(1+\frac{r}{p-1}\Big)\Big(\int_{D}\sum_{j> r}\sigma_j\big(A(t)\big)^2\,\mathrm{d}t\Big).
\end{equation}\end{theorem}
\begin{proof}
{Consider the measure spaces $(D,\mathcal{B}({D}),\lambda)$ and $(\mathbb{R}^{n\times (r+p)},\mathcal{B}({\mathbb{R}^{n\times (r+p)}}),\mu_{\Omega})$ from Lemma~\ref{measurable}. Since $\lambda(D)<\infty$ (because $D$ is compact) and $\mu_{\Omega} (\mathbb{R}^{n\times (r+p)})=1$, both spaces are $\sigma$-finite}. By Lemma \ref{measurable}, $\|(I-\mathcal{P}_{A(t)\Omega})A(t)\|^2_F$ is a non-negative measurable function and thus, by Tonelli's theorem, we conclude that
\begin{align*}
     \mathbb{E}\Big[\int_{D}\|(I-\mathcal{P}_{A(t)\Omega})A(t)\|^2_F\,\mathrm{d}t\Big]
     &=\int_{D}\mathbb{E}[\|(I-\mathcal{P}_{A(t)\Omega})A(t)\|^2_F]\,\mathrm{d}t
     \\&
     \leq \int_{D}\Big(1+\frac{r}{p-1}\Big) \sum_{j> r}\sigma_j\big(A(t)\big)^2 \,\mathrm{d}t,
\end{align*}
where the last inequality follows from Theorem \ref{HMT}.
\end{proof}
The second factor in the right-hand side of~\eqref{hmtexp-bound} is the (squared) $L^2$ approximation error obtained when computing the best rank-$r$ approximation for each $t$ with the SVD. In comparison, the expected error of the much simpler approximation~\eqref{eq:parahmt} is only moderately larger. When $A(t) \equiv B$ for a constant matrix $B \in \R^{m\times n}$, then Theorem \ref{l2-exp} coincides with Theorem \ref{HMT}.

Our tail bound will be derived from bounds on the higher-order moments of the approximation error. For this purpose, we need to control higher-order moments of the Frobenius norm of the pseudo-inverse of a Gaussian random matrix. For a random variable $Z$, we let $\mathbb{E}^q(Z) := (\mathbb{E}|Z|^q)^{1/q}$ for $q\in \mathbb R$ with $q\geq 1$.

%\textcolor{red}{Double check to which extent this result is contained in [Rosen, Dietrich von (1988). "Moments for the Inverted Wishart Distribution". Scand. J. Stat. 15: 97--109]}
\begin{lemma}
\label{Lq_psudo}
Let $\Omega$ be an $r\times (r+p)$ Gaussian random matrix with 
 $p\geq4$ and set $q=p/2$. Then
\begin{equation*}
    \mathbb{E}^q[\|\Omega^\dagger\|_F^2]\leq r\left(\frac{\Gamma(\frac{1}{2})}{2^q\Gamma(q+\frac{1}{2})}\right)^{\frac{1}{q}}
\end{equation*} 
\end{lemma}
\begin{proof}
The proof follows closely the proof of \cite[Therorem A.7]{halko2011finding}.
First, note that
    \begin{equation*}
       \|\Omega^\dagger\|^2_F=\text{trace}\big[(\Omega^\dagger)^T\Omega^\dagger \big]=\text{trace}\big[(\Omega\Omega^T)^{-1} \big]=\sum_{j=1}^r[(\Omega\Omega^T)^{-1}]_{jj},
    \end{equation*} where the second equality holds almost surely since the Wishart matrix $\Omega\Omega^T$ is invertible with probability 1. Then by \cite[Proposition A.5]{halko2011finding}, $[(\Omega\Omega^T)^{-1}]_{jj} =  X_j^{-1}$ with $X_j\sim \chi^2_{p+1}$ for $j=1,\ldots,r$. Taking the $q$th moment and using the triangle inequality, we get
    \begin{equation*}
        \mathbb{E}^q{[\|\Omega^\dagger\|^2_F]}=\mathbb{E}^q{\big[\sum_{j=1}^rX_j^{-1}\big]}\leq\sum_{j=1}^r\mathbb{E}^q[X_j^{-1}].
    \end{equation*}By \cite[Proposition A.8]{halko2011finding}, 
    \begin{equation*}
    \mathbb{E}^q[X_j^{-1}]=\Big(\frac{\Gamma((p+1)/2-q)}{2^{q}\Gamma((p+1)/2)}\Big)^{\frac{1}{q}}=\Big(\frac{\Gamma(1/2)}{2^{q}\Gamma(q+1/2)}\Big)^{\frac{1}{q}},\quad j = 1,\ldots,r,
    \end{equation*}
     which completes the proof.
    \end{proof}
We also need to control higher-order moments of the Frobenius norm of products involving Gaussian random matrices.
For $q = 1$, the following Lemma coincides with~\cite[Proposition 10.1]{halko2011finding}.
\begin{lemma}
\label{L2norm}
Let $\Omega$ be a Gaussian random matrix. For any $q\geq 1$ and any (fixed) real matrices $S$, $T$ of compatible size, it holds that \begin{equation*}
    \mathbb{E}^q[\|S\Omega T\|_F^2]\leq \left(\frac{2^q\Gamma(\frac{1}{2}+q)}{\Gamma(\frac{1}{2})}\right)^{\frac{1}{q}} \|S\|^2_F\|T\|_F^2.
\end{equation*}
\begin{proof}
Because of the invariance of Gaussian random matrices under orthogonal transformations, we may assume that $S$ and $T$ are diagonal matrices and, hence, 
\begin{equation*}
    \mathbb{E}^q[\|S\Omega T\|_F^2]= \mathbb{E}^q\Big[\sum_{jk}|s_{jj}\Omega_{jk}t_{kk}|^2\Big].
\end{equation*}
Using the triangle inequality,
\begin{equation*}
    \mathbb{E}^q\Big[\sum_{jk}|s_{jj}\Omega_{jk}t_{kk}|^2\Big]\leq \sum_{jk}\mathbb{E}^q\Big[|s_{jj}\Omega_{jk}t_{kk}|^2\Big]=\sum_{jk}|s_{jj}|^2|t_{kk}|^2\mathbb{E}^q\Big[|\Omega_{jk}|^2\Big].
\end{equation*} Note that $|\Omega_{jk}|^2$ is a $\chi^2$ random variable with 1 degree of freedom. Therefore, by \cite[Theorem 3.3.2]{hogg1995introduction},\begin{equation*}
    \mathbb{E}^q\Big[|\Omega_{jk}|^2\Big]{=\left(\frac{2^q\Gamma(\frac{1}{2}+q)}{\Gamma(\frac{1}{2})}\right)^{\frac{1}{q}}},
\end{equation*}
which completes the proof.
\end{proof}
\end{lemma}
We are now ready to establish bounds on the higher-order moments of the approximation error. Using the Markov inequality, we can then derive a tail bound for the approximation error.
\begin{lemma}
\label{pmoment}
Consider $A\in {C}(D,\mathbb{R}^{m\times n})$ and an $n\times (r+p)$ Gaussian random matrix  $\Omega$ with $r\ge 2$, $p \ge 4$. Then
\begin{equation*}
    \mathbb{E}^q\Big[\int_{D}\|(I-\mathcal{P}_{A(t)\Omega})A(t)\|_F^2\,\mathrm{d}t\Big]\leq (1+r)\Big(\int_{D}\sum_{j> r}\sigma_j\big(A(t)\big)^2\,\mathrm{d}t\Big),
\end{equation*} where $q=p/2$.
\end{lemma}
\begin{proof} As in Lemma \ref{measurable}, consider the measure space $(D,\mathcal{B}({D}),\lambda)$ and the probability space $(\mathbb{R}^{n\times (r+p)},\mathcal{B}({\mathbb{R}^{n\times (r+p)}}),\mu_{\Omega})$, which are $\sigma$-finite measure spaces. By Lemma \ref{measurable}, $\|(I-\mathcal{P}_{A(t)\Omega})A(t)\|_F^2$ is measurable on the product measure. This allows us to apply Minkowski's integral inequality~\cite[p.194]{folland1999real} and obtain \begin{align*}
    \mathbb{E}^q\Big(\int_{D}\|(I-\mathcal{P}_{A(t)\Omega})A(t)\|_F^2\,\mathrm{d}t\Big)&=\Big(\mathbb{E}\Big(\int_{D}\|(I-\mathcal{P}_{A(t)\Omega})A(t)\|_F^2 \,\mathrm{d}t\Big)^q\Big)^{\frac{1}{q}}\\&\leq \int_{D}\mathbb{E}^q(\|(I-\mathcal{P}_{A(t)\Omega})A(t)\|_F^2)\,\mathrm{d}t.
\end{align*}
For each fixed $t\in D$, we partition the (pointwise) SVD as in~\eqref{partition}: \begin{equation*}
    A(t)=U(t)\begin{bmatrix}
    \Sigma_1(t)& \\
    &\Sigma_2(t) \\
   
    \end{bmatrix}V(t)^T=U(t)
    \begin{bmatrix}
    \Sigma_1(t)& \\
    &\Sigma_2(t) \\

    \end{bmatrix}\begin{bmatrix}
    V_1(t)^T \\
    V_2(t)^T \\
 
    \end{bmatrix},
\end{equation*} where $\Sigma_1(t)\in \mathbb{R}^{r\times r}$, $\Sigma_2(t)\in \mathbb{R}^{(m-r)\times (n-r)}$, $V_1(t)^T\in \mathbb{R}^{r\times n}$ and $V_2(t)^T\in \mathbb{R}^{(n-r)\times n}$. Let $\Omega_1(t):=V_1(t)^T\Omega\in \mathbb{R}^{r\times (r+p)}$ and $\Omega_2(t):=V_2(t)^T\Omega\in \mathbb{R}^{(n-r)\times (r+p)}$.
By Theorem \ref{deter bound}, 
\begin{equation} \label{eq:bound}
    \|(I-\mathcal{P}_{A(t)\Omega})A(t)\|_F^2\leq \|\Sigma_2(t)\|_F^2+\|\Sigma_2(t)\Omega_2(t)\Omega_1(t)^\dagger\|_F^2.
\end{equation}
For \emph{fixed} $t$, we have that $\Omega_1(t)$ and $\Omega_2(t)$ are independent Gaussian random matrices and $\Omega_1(t)$ has almost surely full row rank. To bound the higher moments of the second term in~\eqref{eq:bound}, we use the law of total expectation as well as Lemma \ref{L2norm} and Lemma \ref{Lq_psudo} to obtain
\begin{align*}
    \mathbb{E}^q(\|\Sigma_2(t)\Omega_2(t)\Omega_1(t)^\dagger\|_F^2)&\leq \mathbb{E}^q(\|\Sigma_2(t)\Omega_2(t)\Omega_1(t)^\dagger\|_F^2)\\&=\Big[\mathbb{E}(\|\Sigma_2(t)\Omega_2(t)\Omega_1(t)^\dagger\|_F^{2q})\Big]^\frac{1}{q}\\
    &=\Big[\mathbb{E}\big[\mathbb{E}\big(\|\Sigma_2(t)\Omega_2(t)\Omega_1(t)^\dagger\|_F^{2q}\big|\Omega_1(t)\big)\big]\Big]^\frac{1}{q}\\
    &\leq \left(\frac{2^q\Gamma(\frac{1}{2}+q)}{\Gamma(\frac{1}{2})}\right)^{\frac{1}{q}} \cdot\Big[\mathbb{E}\big[\|\Sigma_2(t)\|_F^{2q}\|\Omega_1(t)^\dagger\|_F^{2q}\big]\Big]^\frac{1}{q}\\
    &=\left(\frac{2^q\Gamma(\frac{1}{2}+q)}{\Gamma(\frac{1}{2})}\right)^{\frac{1}{q}}\cdot\|\Sigma_2(t)\|_F^{2}\Big[\mathbb{E}\big[\|\Omega_1(t)^\dagger\|_F^{2q}\big]\Big]^\frac{1}{q}
    \\
    &\leq r\left(\frac{2^q\Gamma(\frac{1}{2}+q)}{\Gamma(\frac{1}{2})}\right)^{\frac{1}{q}}\left(\frac{\Gamma(\frac{1}{2})}{2^q\Gamma(q+\frac{1}{2})}\right)^{\frac{1}{q}}\cdot\|\Sigma_2(t)\|_F^{2}
    \\&= r\cdot\|\Sigma_2(t)\|_F^{2}.
\end{align*} In summary, we get
\begin{align} 
    \mathbb{E}^q\Big(\int_{D}\|(I-\mathcal{P}_{A(t)\Omega})A(t)\|_F^2\,\mathrm{d}t\Big)&\leq \int_{D}\mathbb{E}^q(\|(I-\mathcal{P}_{A(t)\Omega})A(t)\|_F^2)\,\mathrm{d}t\nonumber \\ &\leq\int_{D}\mathbb{E}^q(\|\Sigma_2(t)\|_F^2+\|\Sigma_2(t)\Omega_2(t)\Omega_1(t)^\dagger\|_F^2)\,\mathrm{d}t\nonumber \\&\leq(1+r)\int_{D}\|\Sigma_2(t)\|_F^2\,\mathrm{d}t. \label{eq:lp-inequaility}
\end{align}
\end{proof}

\begin{theorem}
\label{l2-tail}
Consider $A\in {C}(D,\mathbb{R}^{m\times n})$ and an $n\times (r+p)$ Gaussian random matrix  $\Omega$ with $r\ge 2$, $p \ge 4$. Then for all $\gamma \geq 1$, the probability that the inequality
\begin{equation*}
   \Big( \int_{D}\|(I-\mathcal{P}_{A(t)\Omega})A(t)\|_F^2\,\mathrm{d}t\Big)^{\frac{1}{2}}< \gamma \cdot\sqrt{1+r} \Big(\int_{D}\sum_{j> r}\sigma_j\big(A(t)\big)^2\,\mathrm{d}t\Big)^{\frac{1}{2}},
\end{equation*} fails is at most $\gamma^{-p}$.
\end{theorem}
\begin{proof}
Let $q=p/2$. By the Markov inequality and Lemma \ref{pmoment}, for any $\gamma\geq 1$,
\begin{align*}
\frac{1}{\gamma^p}&\geq\Pr\Big\{\int_{D}\|(I-\mathcal{P}_{A(t)\Omega})A(t)\|_F^2\,\mathrm{d}t\geq \gamma^2\cdot \mathbb{E}^q\Big[\int_{D}\|(I-\mathcal{P}_{A(t)\Omega})A(t)\|_F^2\,\mathrm{d}t\Big]\Big\}\\
&\geq\Pr\Big\{\int_{D}\|(I-\mathcal{P}_{A(t)\Omega})A(t)\|_F^2\,\mathrm{d}t\geq \gamma^2\cdot(1+r)\int_{D}\|\Sigma_2(t)\|_F^2\,\mathrm{d}t\Big\}\\
&= \Pr\Big\{\Big(\int_{D}\|(I-\mathcal{P}_{A(t)\Omega})A(t)\|_F^2\,\mathrm{d}t\Big)^{\frac{1}{2}}\geq \gamma\cdot\sqrt{1+r}\Big(\int_{D}\|\Sigma_2(t)\|_F^2\,\mathrm{d}t\Big)^{\frac{1}{2}}\Big\}.
\end{align*}
\end{proof}
Theorem \ref{l2-tail} shows that the $L^2$ approximation error deviates, with high probability, by not much more than a factor $\sqrt{1+r}$ from the best rank-$r$ approximation error. Compared to Theorem \ref{HMT tail} when $A(t)$ is a constant matrix, the bound of Theorem \ref{l2-tail} is simpler (because it does not feature the second term in~\eqref{eq:hmttail}) and slightly worse (because the factor $\sqrt{1+r}$ cannot be mitigated by choosing $p$ sufficiently large). A lot of the simplicity comes from the fact that the $L^2$ norm is used on $D$. 
%\begin{equation*}
%    \|(I-\mathcal{P}_{B\Omega})B\|_F\leq \Big(1+\gamma\cdot \sqrt{1+r}\Big)\Big(\sum_{j>r}\sigma_j(B)^2\Big)^{\frac{1}{2}},
%\end{equation*} fails is at most $\gamma^{-p}$. This is surprisingly close to Theorem \ref{HMT tail} if we ignore the contribution of $s$. However, increasing the oversampling parameter $p$ now only lower the failure probability.

\subsubsection{Uniform norm error}
Throughout this section, we will assume that $D=[0,T]\subset \mathbb{R}$.
To derive a tail bound for the error in the uniform norm, we will discretize the interval $[0,T]$ and control the error on each sub-interval. To achieve this, we first provide a structural bound for the (standard) HMT approximation error when the projection is perturbed. The bound was derived in~\cite{Connolly2023Probabilistic}.
\begin{lemma}
\label{RSVD pertubation}
Consider $B\in \mathbb{R}^{m\times n}$ and with the notation introduced in \eqref{partition}. Suppose that $V_1^T\Omega\in \mathbb{R}^{r\times (r+p)}$ has full row rank. Then for any $E\in \mathbb{R}^{m\times n}$,
    \begin{equation*}
        \|(I-\mathcal{P}_{(B+E)\Omega})B\|_F\leq \sqrt{ \|\Sigma_2\|^2_F+\|\Sigma_2V_2^T\Omega (V_1^T\Omega)^\dagger \|^2_F}+\|EV_2(V_2^T\Omega) (V_1^T\Omega)^\dagger\|_F+\|E\|_F.
    \end{equation*}
\end{lemma}
% \begin{proof}
%      Note that 
%     \begin{equation*}
%         \|(I-\mathcal{P}_{(B+E)\Omega})B\|_F\leq\|(I-\mathcal{P}_{(B+E)\Omega})B\Omega X\|_F+\|(I-\mathcal{P}_{(B+E)\Omega})B(I-\Omega X)\|_F,
%     \end{equation*} for an arbitrary matrix $X$ of suitable size. For the first term, observe that 
%     \begin{align*}
%         \|(I-\mathcal{P}_{(B+E)\Omega})B\Omega X\|_F&= \|(I-\mathcal{P}_{(B+E)\Omega})[(B+E)\Omega X-E\Omega X]\|_F\\ &=\|(I-\mathcal{P}_{(B+E)\Omega})E\Omega X\|_F\\
%         &\leq \|E\Omega X\|_F.
%     \end{align*}
% For the second term, one can choose $X=(V_1^T\Omega)^\dagger V_1^T$. Since $V_1^T\Omega (V_1^T\Omega)^\dagger=I$, we obtain
% \begin{align*}
%     \|(I-\mathcal{P}_{(B+E)\Omega})B(I-\Omega X)\|^2_F&\leq \|B(I-\Omega (V_1^T\Omega)^\dagger V_1^T)\|^2_F\\ &= \|B(I-V_1 V_1^T)(I-\Omega (V_1^T\Omega)^\dagger V_1^T)\|^2_F\\ &= \|B(I-V_1 V_1^T)\|^2_F+\|B(I-V_1V_1^T)(\Omega (V_1^T\Omega)^\dagger V_1^T)\|^2_F\\ &= \|\Sigma_2\|^2_F+\|B(I-V_1V_1^T)\Omega (V_1^T\Omega)^\dagger \|^2_F
% \end{align*} Combine the two inequalities, we have 
% \begin{align*}
%      \|(I-\mathcal{P}_{(B+E)\Omega})B\|_F&\leq\sqrt{ \|\Sigma_2\|^2_F+\|B(I-V_1V_1^T)\Omega (V_1^T\Omega)^\dagger \|^2_F}+\|E\Omega (V_1^T\Omega)^\dagger V_1^T\|_F\\
%     & = \sqrt{ \|\Sigma_2\|^2_F+\|\Sigma_2V_2^T\Omega (V_1^T\Omega)^\dagger \|^2_F}+\|E(I-V_1V_1^T+V_1V_1^T)\Omega (V_1^T\Omega)^\dagger V_1^T\|_F\\
%     & \leq \sqrt{ \|\Sigma_2\|^2_F+\|\Sigma_2V_2^T\Omega (V_1^T\Omega)^\dagger \|^2_F}+\|EV_2(V_2^T\Omega) (V_1^T\Omega)^\dagger\|_F+\|E\|_F.
% \end{align*}
% \end{proof}
To proceed, we control the Frobenius norm of a parameter-dependent matrix multiplied by a Gaussian matrix.  
\begin{theorem}
\label{Thm:supnorm}
Let $\Omega\in \mathbb{R}^{n\times (r+p)}$ be a Gaussian random matrix and $B$ be a fixed matrix of compatible size.
Suppose $E(t)\in {C}([t_0,t_0+T],\mathbb{R}^{m\times n})$ such that $E(t_0)=0$ and $E(t)$ is Lipschitz continuous with a Lipschitz constant $L\geq0$  in the Frobenius norm. Then for any $u\geq 3$,
    \begin{equation*}
        \Pr \Big\{\sup_{t\in[t_0,t_0+T]}\|E(t)\Omega B\|_F> 2LT\|B\|_F (1+4u)\Big\} \leq e^{-u^2/2}.
    \end{equation*}
\end{theorem}
\begin{proof}
    The following proof is based on chaining \cite{adler2007random,Talagrand2021}. We first let $X(t)=E(t)\Omega B$.
    \begin{enumerate}
        \item \emph{Increment probability.} For any fixed $s,t\in [t_0,t_0+T]$, We first control the norm of $X(t)-X(s)$. Since $\Omega$ is a Gaussian matrix, by the concentration for functions of a Gaussian matrix \cite[Proposition 10.3]{halko2011finding}, we have
        \begin{equation*}
            \Pr\{\|X(t)-X(s)\|_F\geq \mathbb{E}[\|(E(t)-E(s))\Omega B\|_F]+u\cdot \|(E(t)-E(s))\|_F\|B\|_F\}\leq e^{-u^2/2}.
        \end{equation*}
        Note that \begin{align*}
            \mathbb{E}[\|(E(t)-E(s))\Omega B\|_F]\leq& \Big(\mathbb{E}[\|(E(t)-E(s))\Omega B\|^2_F]\Big)^\frac{1}{2}\\& =\|(E(t)-E(s))\|_F\|B\|_F\\& \leq L\|B\|_F|t-s|.
        \end{align*}
        Hence,
        \begin{equation*}
            \Pr\Big\{\|X(t)-X(s)\|_F\geq  L\|B\|_F|t-s|(1+u)\Big\}\leq e^{-u^2/2}.
        \end{equation*}
        \item \emph{Partition.} For $n\geq 0$, we consider a subset $T_n\subset [t_0,t_0+T]$ as a uniform partition of $[t_0,t_0+T]$ with $1$ point for $n=1$ $(T_0=\{t_0\})$ and with $2^{2^n}$ points for $n\geq 1$, and we consider $\pi_n(t)\in T_n$ be the point that is the closest point to $t$ among $T_n$. Note that 
        \begin{equation*}
            X(t)-X(t_0)=\sum_{n\geq 1}\big(X(\pi_n(t))-X(\pi_{n-1}(t))\big).
        \end{equation*} The infinite sum is well-defined and converges with probability one, see \cite[p.16]{adler2007random}. By triangle inequality,
        \begin{equation*}
            \|X(t)-X(t_0)\|_F\leq \sum_{n\geq 1}\|X(\pi_n(t))-X(\pi_{n-1}(t))\|_F.
        \end{equation*} We then control each term of the sum.
        \item \emph{Controlling each interval.} We define the event $I_u$ by
        \begin{equation*}
            \forall n\geq 1,\;\; \forall t,\;\; \|X(\pi_n(t))-X(\pi_{n-1}(t))\|_F\leq  L|\pi_n(t)-\pi_{n-1}(t)|\|B\|_F (1+2^{n/2}u).      \end{equation*} Since the number of possible pairs  $(\pi_n(t),\pi_{n-1}(t))$ is at most $$\text{card}(T_{n-1})\cdot \text{card}(T_n)=2^{2^n}\cdot 2^{2^{n-1}}\leq 2^{2^{n+1}},$$ it follows from the  union bound that
            \begin{equation*}
                \Pr(I_u^c)\leq \sum_{n\geq1} 2^{2^{n+1}} e^{-(2^{n-1}u^2)}.
            \end{equation*} For $n\geq 1$ and $u\geq 3$, we obtain
            \begin{equation*}
                u^2 2^{n-1}\geq \frac{u^2}{2}+u^2 2^{n-2}\geq \frac{u^2}{2}+2^{n+1}.            \end{equation*}
                Hence $$\Pr(I_u^c)\leq e^{-u^2/2}\sum_{n\geq1} \frac{2^{2^{n+1}}}{e^{2^{n+1}}} \leq e^{-u^2/2}\sum_{n\geq1} \frac{1}{2^n} \leq e^{-u^2/2} .$$
             Under the event $I_u$
            \begin{align*}
               \sup_{t\in [t_0,t_0+T]} \|X(t)-X(t_0)\|_F&\leq \sup_{t\in [t_0,t_0+T]}\sum_{n\geq 1}\|X(\pi_n(t))-X(\pi_{n-1}(t))\|_F\\&\leq \sup_{t\in [t_0,t_0+T]}\sum_{n\geq 1}L|\pi_n(t)-\pi_{n-1}(t)|\|B\|_F (1+2^{n/2}u)\\&\leq \sup_{t\in [t_0,t_0+T]}\sum_{n\geq 1}L(|\pi_n(t)-t|+|t-\pi_{n-1}(t)|)\|B\|_F (1+2^{n/2}u)\\&\leq \sup_{t\in [t_0,t_0+T]}L\|B\|_F\sum_{n\geq 1}(2|t-\pi_{n-1}(t)|) (1+2^{n/2}u)\\&\leq 2L\|B\|_F\sum_{n\geq 1}\frac{T}{2^{2^{n-1}}}(1+2^{n/2}u)\\&\leq 2LT\|B\|_F(1+\sum_{n\geq 1}\frac{2^{n/2}}{2^{2^{n-1}}}u).
            \end{align*}
            Finally, we have 
            \begin{equation*}
                \Pr\Big\{ \sup_{t\in [t_0,t_0+T]} \|X(t)-X(t_0)\|_F>2LT\|B\|_F(1+4u)\Big\}\leq e^{-u^2/2}.
            \end{equation*}
    \end{enumerate}
\end{proof}
Instead of using the submultiplicative of Frobenius norm and directly obtain bound by controlling $\|E(t)\|_F\|\Omega B\|_2$, the technique we used in Theorem \ref{Thm:supnorm} provides a result that is independent of the dimension of $\Omega$.

Now, we are ready to establish bounds on the uniform norm approximation error.
\begin{theorem}
\label{sup-tail}
Consider $A\in {C}([0,T],\mathbb{R}^{m\times n})$ and an $n\times (r+p)$ Gaussian random matrix  $\Omega$ with $r\ge 2$, $p \ge 4$. Suppose $A(t)$ is Lipschitz continuous with a Lipschitz constant $L\geq0$  in the Frobenius norm. Then for all $\gamma,k \geq 1$ and $u\geq 3$, the probability that the inequality
\begin{align}
\label{eq:sup tail}
  \sup_{t\in [0,T]}&\|(I-\mathcal{P}_{A(t)\Omega})A(t)\|_F\nonumber \\&\quad\leq \Big(1+\gamma \sqrt{\frac{3r}{p+1}}\Big)\sup_{t\in [0,T]}\Big(\sum_{j>r}\sigma_j\big(A(t)\big)^2\Big)^{\frac{1}{2}}+u\gamma\frac{e\sqrt{r+p}}{p+1} \sup_{t\in [0,T]}\sigma_{r+1}\big(A(t)\big)\nonumber\\&\quad\quad+\frac{2TL}{k}\Big[\Big(1+\gamma \sqrt{\frac{3r}{p+1}}\Big)(1+4u)+1\Big]
\end{align} fails is at most $2k(\gamma^{-p}+e^{-u^2/2})$.
\end{theorem}
 \begin{proof} We first create a uniform partition of the time interval $[0,T]$: \begin{equation*}
    0=t_0\leq t_1\leq...\leq t_{k}=T.
\end{equation*}
For $u,\gamma\geq 1$, we define the event $E_{\gamma,u}$ as follows: The error bound of Theorem \ref{HMT tail} holds for every $A(t_i)$
and the bound \begin{equation*}
    \|(V_1(t_i)^T\Omega)^\dagger\|_F\leq \Big(1+\gamma\cdot \sqrt{\frac{3r}{p+1}}\Big).
\end{equation*}
holds for $i = 0,\ldots,k-1$. As the latter bound is actually established and used in the proof of Theorem \ref{HMT tail} from~\cite{halko2011finding}, we directly obtain the following result for the failure probability from combining Theorem \ref{HMT tail} with the union bound:
% are controlled on these points: For any $i=0,\ldots, k-1$, \begin{align*}
%     \sqrt{ \|\Sigma_2(t_i)\|^2_F+\|\Sigma_2(t_i)(V_2(t_i)^T\Omega) (V_1(t_i)^T\Omega)^\dagger \|^2_F}&\leq \Big(1+\gamma\cdot \sqrt{\frac{3r}{p+1}}\Big)\Big(\sum_{j>r}\sigma_j\big(A(t_i)\big)^2\Big)^{\frac{1}{2}}\\&\quad+u\gamma\cdot\frac{e\sqrt{r+p}}{p+1}\cdot \sigma_{r+1}\big(A(t_i)\big).
% \end{align*} By the union bound and Theorem \ref{HMT tail},
\begin{equation*}
    \Pr(E_{\gamma,u}^c)\leq k(2\gamma^{-p}+e^{-u^2/2}).
\end{equation*}
% The proof of Theorem \ref{HMT tail} proceeds by establishing the Frobenius norm of the matrices $V_1(t_i)^T\Omega$ are controlled \cite[p.276]{halko2011finding}, hence we have
% \begin{equation*}
%     \;\;\|(V_1(t_i)^T\Omega)^\dagger\|_F\leq \Big(1+\gamma\cdot \sqrt{\frac{3r}{p+1}}\Big),\;\forall i.
%\end{equation*}
To proceed, we define an event $\Pi_u$ for which the error in each subinterval is controlled: For every $i=0,\ldots k-1$,\begin{equation*}
    \sup_{t\in[t_i,t_i+\frac{T}{k}]}\big\|\big(A(t)-A(t_i)\big)V_2(t_i)\big(V_2(t_i)^T\Omega\big) \big(V_1(t_i)^T\Omega\big)^\dagger \big\|_F\leq \frac{2LT}{k}\Big(1+\gamma\cdot \sqrt{\frac{3r}{p+1}}\Big) (1+4u).
\end{equation*} As $V_1(t_i)\Omega$ and $V_2(t_i)\Omega$ are independent, to bound $\Pr(\Pi_u^c)$, we can use Theorem \ref{Thm:supnorm} with $E(t)=\big(A(t)-A(t_i)\big)V_2(t_i)$, $B=\big(V_1(t_i)^T\Omega\big)^\dagger$ and the union bound. We conclude
\begin{equation*}
     \Pr(\Pi_u^c)\leq ke^{-u^2/2}.
\end{equation*}

For any $t$, if we take the closest point $t_i$ and by Lemma \ref{RSVD pertubation},
\begin{align*}
    \|(I-\mathcal{P}_{A(t)\Omega})A(t)\|_F&\leq \|(I-\mathcal{P}_{A(t)\Omega})A(t_i)\|_F+\|(I-\mathcal{P}_{A(t)\Omega})\big(A(t)-A(t_i)\big)\|_F\\&\leq \sqrt{ \|\Sigma_2(t_i)\|^2_F+\|\Sigma_2(t_i)(V_2(t_i)^T\Omega) (V_1(t_i)^T\Omega)^\dagger \|^2_F}\\&\quad+\|\big(A(t)-A(t_i)\big)V_2(t_i)(V_2(t_i)^T\Omega) (V_1(t_i)^T\Omega)^\dagger\|_F+2\|A(t)-A(t_i)\|_F.\\&\leq \sqrt{ \|\Sigma_2(t_i)\|^2_F+\|\Sigma_2(t_i)(V_2(t_i)^T\Omega) (V_1(t_i)^T\Omega)^\dagger \|^2_F}\\&\quad+\sup_{t\in[t_i,t_i+\frac{T}{k}]}\|\big(A(t)-A(t_i)\big)V_2(t_i)(V_2(t_i)^T\Omega) (V_1(t_i)^T\Omega)^\dagger\|_F+2\frac{LT}{k}.
    \end{align*}
Under the event $\Pi_u,E_{\gamma,u}$, we obtain,
\begin{align*}
    \|(I-\mathcal{P}_{A(t)\Omega})A(t)\|_F&\leq \Big(1+\gamma\cdot \sqrt{\frac{3r}{p+1}}\Big)\Big(\sum_{j>r}\sigma_j\big(A(t_i)\big)^2\Big)^{\frac{1}{2}}+u\gamma\cdot\frac{e\sqrt{r+p}}{p+1}\cdot \sigma_{r+1}\big(A(t_i)\big)\\&\quad+\frac{2LT}{k}\Big[\Big(1+\gamma\cdot \sqrt{\frac{3r}{p+1}}\Big)(1+4u)+1\Big]\\& \leq \Big(1+\gamma\cdot \sqrt{\frac{3r}{p+1}}\Big)\sup_{t\in [0,T]}\Big(\sum_{j>r}\sigma_j\big(A(t)\big)^2\Big)^{\frac{1}{2}}+u\gamma\frac{e\sqrt{r+p}}{p+1} \sup_{t\in [0,T]}\sigma_{r+1}\big(A(t)\big)\\&\quad+\frac{2LT}{k}\Big[\Big(1+\gamma\sqrt{\frac{3r}{p+1}}\Big)(1+4u)+1\Big]
\end{align*} and by union bound the event fails with probability at most $2k(\gamma^{-p}+e^{-u^2/2})$.
 \end{proof}
Theorem \ref{sup-tail} shows that the uniform norm approximation error deviates by not much more than some factors from the best rank-$r$ approximation error plus an extra term with high probability. Denote $\tau=\sup_{t\in [0,T]}\Big(\sum_{j>r}\sigma_j\big(A(t)\big)^2\Big)^{\frac{1}{2}}$. When $p=O(r)$, we can expect that the first and the third terms in \eqref{eq:sup tail} dominate the bound. By choosing $k=O(\frac{1}{\tau})$ and $\gamma=\frac{s}{\tau^{1/p}}$ for some $s>\tau^{1/p}$, those terms are $O(s\tau^{1-\frac{1}{p}})$ with probability $O(s^{-p})$. This shows that when $p$ is well-chosen, the method is also robust in uniform norm. 

\section{The generalized Nystr\"{o}m method for parameter-dependent matrices}
Before we describe and analyze the generalized Nystr\"{o}m method for parameter-dependent
matrices, we recall basic results of the generalized Nystr\"{o}m method from \cite{tropp2017practical} for a constant matrix $B\in \mathbb{R}^{m\times n}$.

\subsection{Generalized Nystr\"{o}m method for a constant matrix}

Given random DRMs $\Omega \in \mathbb{R}^{n\times (r+p)}$ and $\Psi \in \mathbb{R}^{m\times (r+p+\ell)}$, the generalized Nystr\"{o}m method constructs a low-rank approximation by performing an oblique projection of the columns of $B$ onto $\mathrm{span}(B\Omega)$:
\begin{equation*}
    B\approx \mathcal{P}_{B\Omega,\Psi}B, \quad \mathcal{P}_{B\Omega,\Psi}:=B\Omega(\Psi^TB\Omega)^{\dagger}\Psi^T.
\end{equation*}
When $\Psi^TB\Omega$ has full column rank, one gets the equivalent expression
\begin{equation} \label{eq:nystromprojector}
 \mathcal{P}_{B\Omega,\Psi}B=Q(\Psi^TQ)^\dagger \Psi^TB,
\end{equation}
where $Q$ is an orthonormal basis of $\mathrm{span}(B\Omega)$ computed, e.g., via an economy-size QR factorization $QR=B\Omega$.
The analysis of this approximation relies on the following structural bound.
\begin{theorem}[{\cite[Lemmas A3 and A4]{tropp2017practical}}]
\label{Ny_Structural}
    With the notation introduced above, assume that $\Psi^TB\Omega$ has full column rank. Then
    \begin{equation*}
        \|(I-\mathcal{P}_{B\Omega,\Psi})B\|^2_F=\|(I-QQ^T)B\|_F^2+\|(\Psi^TQ)^\dagger(\Psi^TQ_\perp) (Q^T_\perp B)\big\|_F^2.
    \end{equation*}
    \end{theorem}
\noindent The first term of the sum above is just the HMT error analyzed in Theorem~\ref{HMT}. The second term can also be expected to stay close to the HMT error provided that $\Psi^TQ$ is not too ill-conditioned. These considerations lead to the following bound on the expected error.
\begin{theorem}[{\cite[Theorem 4.3]{tropp2017practical}}]
\label{Ny_error}
Let $\sigma_1(B)\geq \sigma_2(B)\geq\cdots$ denote the singular values of $B\in \mathbb{R}^{m\times n}$. Suppose that $\Omega\in \mathbb{R}^{n\times (r+p)}$ and $\Psi\in \mathbb{R}^{m\times (r+p+\ell)}$ are independent Gaussian random matrices, where $r\geq2$, $p\geq2$ and $\ell\geq 2$.
Then
\begin{equation*}
    \mathbb{E}\|(I-\mathcal{P}_{B\Omega,\Psi})B\|^2_F= \Big({1+\frac{r+p}{\ell-1}}\Big)\mathbb{E}\|(I-\mathcal{P}_{B\Omega})B\|^2_F\leq\Big({1+\frac{r+p}{\ell-1}}\Big)\Big(1+\frac{r}{p-1}\Big)\Big(\sum_{j>r}\sigma_j(B)^2\Big).
\end{equation*}
\end{theorem}

The use of the pseudoinverse in~\eqref{eq:nystromprojector} potentially introduces instabilities when the matrix $\Psi^TB\Omega$ is ill-conditioned; see~\cite{nakatsukasa2020fast} for a discussion and numerically stable implementation. The computation of the sketches $B\Omega$ and $\Psi^TB$ typically dominates the overall computational effort, especially when $B$, $\Omega$ and $\Psi$ are dense unstructured matrices.
%, the sketch requires $O(mn(r+p+\ell))$ operations while the overall cost of the generalized Nystr\"{o}m method is $O(mn(r+p+\ell)+(r+p+\ell)(r+p)^2)$. 

\subsection{Generalized Nystr\"{o}m method for a parameter-dependent matrix}

Our extension of the generalized Nystr\"{o}m method to a parameter-dependent matrix $A \in C(D,\mathbb{R}^{m\times n})$ takes the form 
\begin{equation}
\label{eq:paragn}
    A(t)\approx \mathcal{P}_{A(t)\Omega,\Psi}A(t)=\big(A(t)\Omega\big) \big(\Psi^TA(t)\Omega\big)^\dagger\big(\Psi^TA(t)\big)\;\;\;\text{for all }t\in D,
\end{equation} where $\Psi\in \mathbb{R}^{n\times (r+p+\ell)}$ and $\Omega\in \mathbb{R}^{n\times (r+p)}$ are \emph{constant} DRMs.

As in the constant matrix case, some care is needed in the implementation of the generalized Nystr\"{o}m method to avoid numerical instability. Algorithm~\ref{Algro: GN_para} follows the implementation of the generalized Nystr\"{o}m method proposed in \cite{nakatsukasa2020fast}. 
\begin{algorithm}[ht]
\caption{{Generalized Nystr\"{o}m method for $A(t)$} }\label{Algro: GN_para}
\algorithmicrequire $A\in \mathcal{C}(D,\mathbb{R}^{m\times n})$ , integers $r+p, \ell>0$ and evaluation points $t_1,\ldots,t_q \in D$.\\
\algorithmicensure  Matrices $Q_{t_j}, W_{t_j}$ defining low-rank approximations $A(t_j) \approx Q_{t_j} W_{t_j}^T$ for $j=1,\ldots,q$.
     \begin{algorithmic}[1]
     \State Generate DRMs $\Omega\in \mathbb{R}^{n\times (r+p)}$ and $\Psi\in \mathbb{R}^{m\times (r+p+\ell)}$.
        \For{$j=1,\ldots,q$}
      \State Compute $X_{t_j} = A(t_j)\Omega$ and $Y_{t_j} = \Psi^TA(t_j)$.
      \State Compute economy-size QR factorization $\Tilde{Q}_{t_j}\Tilde{R}_{t_j}=\Psi^TX_{t_j}$.
      \State Compute $Q_{t_j}=X_{t_j}(\Tilde{R}_{t_j})^\dagger_\epsilon$ and $W_{t_j}=Y_{t_j}^T\Tilde{Q}_{t_j}$.
    \EndFor
     \end{algorithmic}
\end{algorithm}%
In Algorithm \ref{Algro: GN_para}, $R^\dagger_\epsilon$ denotes the $\epsilon$-pseudoinverse of a matrix $R$: Consider an SVD
\begin{equation*}
    R=\begin{bmatrix}
U_1 & U_2
\end{bmatrix}\begin{bmatrix}
 \Sigma_1 & \\
 &\Sigma_2
\end{bmatrix}\begin{bmatrix}
V_1 & V_2
\end{bmatrix}^T,
\end{equation*} where $\Sigma_2$ contains all singular values less than $\epsilon$. Then  $R^\dagger_\epsilon:=U_1\Sigma_1^{-1} V_1^T$. 

The streaming property of generalized Nystr\"{o}m extends to the parameter-dependent case. In particular, when the matrix undergoes update $A(t)\leftarrow A(t)+B(t)$, one can update the sketch via $A(t)\Omega\leftarrow A(t)\Omega+B(t)\Omega$, $\Psi^TA(t)\leftarrow \Psi^TA(t)+\Psi^TB(t)$ and cheaply recompute the low-rank approximation. % It is worth mentioning that the use of constant (in $t$) DRMs helps facilitate the update of the sketches.

Our method significantly benefits when $A(t)$ admits an affine linear decomposition~\eqref{affine combination}. Similar to Algorithms~\ref{offline phase 2} and~\ref{online phase 2} for HMT, Algorithm~\ref{GN-offline phase 2} and~\ref{GN-online phase 2} realize an offline/online approach to cheaply compute low-rank approximations for many parameter values.

\begin{algorithm}[ht]
\caption{Generalized Nystr\"{o}m method for $A(t)=\sum^k_{i=1}\varphi_i(t)A_i$ (offline phase)}\label{GN-offline phase 2}
\algorithmicrequire Matrices $A_1,\ldots, A_k\in \mathbb{R}^{m \times n}$ and integers $r+p, \ell >0$.\\
\algorithmicensure  Sketches $X_i\in\mathbb{R}^{m\times (r+p)}, Y_i\in \mathbb{R}^{(r+p+\ell)\times n}$ and $Z_i\in \mathbb{R}^{(r+p+\ell)\times(r+p)}$ for $i=1,\ldots,k$.
     \begin{algorithmic}[1]
     \State Generate DRMs $\Omega\in \mathbb{R}^{n\times (r+p)}$ and $\Psi\in \mathbb{R}^{m\times (r+p+\ell)}$.
      \State Compute $X_1 = A_1\Omega,\ldots ,X_k = A_k\Omega$.
       \State Compute $Y_1 = \Psi^TA_1,\ldots ,Y_k = \Psi^TA_k$.
       \State Compute $Z_1 = Y_1\Omega,\ldots ,Z_k = Y_k\Omega$. 
      \end{algorithmic}
\end{algorithm}
\begin{algorithm}[H]
\caption{Generalized Nystr\"{o}m method for $A(t)=\sum^k_{i=1}\varphi_i(t)A_i$ (online phase)}\label{GN-online phase 2}
\algorithmicrequire Matrices $X_i$ $Y_i$, $Z_i$, functions $\varphi_i \in C(D,\mathbb R)$, for $i=1,\cdots, k$, evaluation points $t_1,\ldots,t_q \in D$.\\
\algorithmicensure  Matrices $Q_{t_j}, W_{t_j}$ defining low-rank approximations $A(t_j) \approx Q_{t_j} W_{t_j}^T$ for $j=1,\ldots,q$.
     \begin{algorithmic}[1]
     \For{$j=1,\ldots,q$}
      
      \State Compute economy-size QR factorization $\sum^k_{i=1}\varphi_i(t_j)Z_i=\Tilde{Q}_{t_j}\Tilde{R}_{t_j}$.
      \State Compute $W_{t_j}=\big(\sum^k_{i=1}\varphi_i(t_j)Y_i\big)^T\Tilde{Q}_{t_j}$.
      \State Compute $Q_{t_j}=\big(\sum^k_{i=1}\varphi_i(t_j)X_i\big)(\Tilde{R}_{t_j})^\dagger_\epsilon$.
     \EndFor
     \end{algorithmic}
\end{algorithm}

To estimate the computational cost of Algorithms~\ref{GN-offline phase 2} and~\ref{GN-online phase 2}, we again assume that $p=O(r)$, $\ell=O(r)$, and let $c_A$ denotes an upper bound on the cost of multiplying $A(t)$ or any of the matrices $A_i$ (or their transposes) with a vector. The total cost of the offline phase is \begin{equation*}
O(krc_A+knr^2).
\end{equation*}
On the other hand, the online phase costs 
\begin{equation*}
O\big(q\big(kr(n+m)+r^2(n+m) \big)\big).
\end{equation*}
Compared to applying $q$ times the (standard) generalized Nystr\"{o}m method to each matrix $A(t_j)$ which requires $O(qr(c_a+nr+r^2))$, we can expect there is a cost reduction when $k\ll q$ and $\max\{k,r\}\ll c_a/m$. {When compared to the costs~\eqref{eq:offline} and~\eqref{eq:online} of HMT for an affine linear parameter matrix, the cost of Algorithms~\ref{GN-offline phase 2} and~\ref{GN-online phase 2} has a much more favorable dependence on $k$ because it avoids the large QR decomposition~\eqref{eq:qr}.}

\subsection{Error analysis}

%\textcolor{blue}{For generlized Nystr\"{o}m method, we only provide $L^2$ approximation error.}
We again assume that $D\subset \mathbb{R}^d$ is compact. Assuming that $\Omega$ and $\Psi$ are Gaussian, we first provide a bound on the expected error for the approximation \eqref{eq:paragn}. Following the structure of the proof of Theorem~\ref{l2-exp}, we first show that the approximation error is measurable. 
\begin{lemma}
\label{measurable_Ny}
Let $\Omega\in \mathbb{R}^{n\times(r+p)}$ and $\Psi\in \mathbb{R}^{n\times(r+p+\ell)}$ be independent Gaussian random matrices. Consider the measure space $\big(D,\mathcal{B}({D}),\lambda\big)$ and the probability spaces $\big(\mathbb{R}^{n\times (r+p)}\times \mathbb{R}^{n\times (r+p+\ell)}, \mathcal{B}({\mathbb{R}^{n\times (r+p)}})\otimes \mathcal{B}({\mathbb{R}^{n\times (r+p+\ell)}}), \mu_{\Omega,\Psi}\big)$, where $\mathcal{B}(\cdot)$ denotes the Borel $\sigma$-algebra of a set, $\lambda$ is the Lebesgue measure and $\mu_{\Omega,\Psi}$ is the joint distribution of $(\Omega,\Psi)$. If $A\in C(D,\mathbb{R}^{m\times n})$ then the function $f:D\times\big(\mathbb{R}^{n\times (r+p)}\times \mathbb{R}^{n\times (r+p+\ell)}\big)\rightarrow \mathbb{R}$ defined by
\begin{equation*}
    f(t,Y,X):=\|(I-\mathcal{P}_{A(t)Y,X})A(t)\|^2_F,\;\;\;
\end{equation*} is measurable on the product measure.
\begin{proof}
For $k\geq1$, we define 
\begin{equation*}
     f_k(t,Y,X):=\Big\|\Big[I-A(t)Y \big(Y^TA(t)^TXX^TA(t)Y+\frac{1}{k}I\big)^{-1}Y^TA(t)^TXX^T \Big]A(t)\Big\|^2_F,
\end{equation*} which is continuous with respect to $t$, $X$ and $Y$. As in the proof of Lemma \ref{measurable}, each $f_k:D\times\big(\mathbb{R}^{n\times (r+p)}\times \mathbb{R}^{n\times (r+p+\ell)}\big)\rightarrow \mathbb{R}$ is a Carath\'eodory function and $f_k(t,Y,X)\stackrel{k\to\infty}{\to} f(t,Y,X)$ pointwise. Hence, $f(t,Y,X)$ is measurable on the product measure.
\end{proof}
\end{lemma}
\begin{theorem}
\label{Ny_error_l2}
Suppose that $A\in {C}(D,\mathbb{R}^{m\times n})$. For a target rank $r\geq 2$ and oversampling parameters $p,\ell \geq 2$, choose independent Gaussian random matrices $\Omega\in \mathbb{R}^{n\times(r+p)}$, $\Psi\in \mathbb{R}^{n\times(r+p+\ell)}$.
Then 
\begin{equation*}
    \mathbb{E}\Big[\int_{D}\|(I-\mathcal{P}_{A(t)\Omega,\Psi})A(t)\|^2_F\,\mathrm{d}t\Big]\leq\Big({1+\frac{r+p}{\ell-1}}\Big)\Big(1+\frac{r}{p-1}\Big)\int_{D}\sum_{j>r}\sigma_j(A(t))^2\,\mathrm{d}t.
\end{equation*}
\end{theorem}
\begin{proof}
By Lemma \ref{measurable_Ny}, $f$ is a non-negative measurable function and thus, as in the proof of Theorem \ref{l2-exp}, Tonelli's theorem and Theorem~\ref{Ny_error} imply that
\begin{align*}
     \mathbb{E}\Big[\int_{D}\|(I-\mathcal{P}_{A(t)\Omega,\Psi})A(t)\|^2_F\,\mathrm{d}t\Big]&=\int_{D}\mathbb{E}[\|(I-\mathcal{P}_{A(t)\Omega,\Psi})A(t)\|^2_F]\,\mathrm{d}t\\&
     \leq \int_{D}\Big({1+\frac{r+p}{\ell-1}}\Big)\Big(1+\frac{r}{p-1}\Big) \sum_{j> r}\sigma_j(A(t))^2\,\mathrm{d}t.
\end{align*}
\end{proof}
Theorem \ref{Ny_error_l2} shows that the expected (squared) $L^2$ approximation error of Algorithm~\ref{Algro: GN_para} stays within a factor $({1+\frac{r+p}{\ell-1}})(1+\frac{r}{p-1})$ of the one obtained when applying the truncated SVD pointwise. Compared to the bound obtained in 
Theorem~\ref{l2-exp} for HMT, only the additional factor $({1+\frac{r+p}{\ell-1}})$ is needed. In the constant case $A(t)\equiv B\in \mathbb{R}^{m\times n}$, the result of Theorem \ref{Ny_error_l2} coincides with Theorem \ref{Ny_error}.

Following the proof of Theorem~\ref{l2-tail}, we obtain a tail bound by establishing bounds on the higher-order moments of the approximation error.
\begin{theorem}
\label{Ny-l2-tail}
In addition to the assumptions of Theorem~\ref{Ny_error_l2}, suppose that $ p\ge 4$ and $ \ell \ge 4$ hold.
Then for all $\gamma\geq 1$, the probability that the inequality
\begin{equation*}
   \Big( \int_{D}\|(I-\mathcal{P}_{A(t)\Omega,\Psi})A(t)\|_F^2\,\mathrm{d}t\Big)^{\frac{1}{2}}< \gamma \cdot\sqrt{(1+r+p)(1+r)}\Big(\int_{D}\sum_{j> r}\sigma_j\big(A(t)\big)^2\,\mathrm{d}t\Big)^{\frac{1}{2}},
\end{equation*} fails is at most $\gamma^{-\min\{p,\ell\}}$.
\end{theorem}
\begin{proof}
Letting $q=\min\{p/2,\ell/2\}$, Minkowski's integral inequality~\cite[p.194]{folland1999real} implies that
 \begin{equation*}
    \mathbb{E}^q\Big(\int_{D}\|(I-\mathcal{P}_{A(t)\Omega,\Psi})A(t)\|_F^2\,\mathrm{d}t\Big)\leq \int_{D}\mathbb{E}^q(\|(I-\mathcal{P}_{A(t)\Omega,\Psi})A(t)\|_F^2)\,\mathrm{d}t.
\end{equation*}
We proceed by bounding the integrand on the right-hand side for \emph{fixed} $t$. Note that we may assume $\rank(A(t))>r+p$ without loss of generality, because otherwise generalized Nystr\"{o}m recovers the exact matrix almost surely and the error as well as all its moments are zero.
This implies that $\Psi^TA(t)\Omega$ has full column rank almost surely.
Considering the QR factorization $A(t)\Omega = Q R$, we denote $\Psi_1 =\Psi^TQ \in \mathbb{R}^{(r+p+\ell)\times (r+p)}$ and $\Psi_2 =\Psi^TQ_\perp\in \mathbb{R}^{(r+p+\ell)\times (n-r-p)}$. Note that we have omitted the dependence on the (fixed) parameter $t$ to simplify the notation.
By Theorem \ref{Ny_Structural},
\begin{equation}\label{eq:bound_ny}
     \mathbb{E}^q(\|(I-\mathcal{P}_{A(t)\Omega,\Psi})A(t)\|_F^2) \leq \mathbb{E}^q\big(\|(I-QQ^T)A(t)\|^2_F\big) +\mathbb{E}^q\big(\|\Psi_1^\dagger \Psi_2 Q^T_\perp A(t)\|^2_F\big).
\end{equation}
We first bound the second term in~\eqref{eq:bound_ny}. Conditioned on $\Omega$, $\Psi_1$ and $\Psi_2$ are independent Gaussian random matrices. Thus, using the law of total expectation and Lemma \ref{L2norm}, we obtain 
\begin{align}
   \mathbb{E}^q\big(\|\Psi_1^\dagger \Psi_2 Q^T_\perp A(t) \|^2_F\big)
   &=\Big( \mathbb{E}_\Omega\big[\mathbb{E}_{\Psi_1}\mathbb{E}_{\Psi_2}\big(\|\Psi_1^\dagger \Psi_2 Q^T_\perp A(t)\|^{2q}_F\big)\big]\Big)^\frac{1}{q} \nonumber \\
   &\leq \left(\frac{2^q\Gamma(\frac{1}{2}+q)}{\Gamma(\frac{1}{2})}\right)^\frac{1}{q}\Big(\mathbb{E}_\Omega\big[\mathbb{E}_{\Psi_1}\big(\|\Psi_1^\dagger\|_F^{2q}\| Q^T_\perp A(t)\|^{2q}_F\big)\big]\Big)^\frac{1}{q} \nonumber \\
   &= \left(\frac{2^q\Gamma(\frac{1}{2}+q)}{\Gamma(\frac{1}{2})}\right)^\frac{1}{q}\Big(\mathbb{E}_\Omega\big[\| Q^T_\perp A(t)\|^{2q}_F\mathbb{E}_{\Psi_1}\big(\|\Psi_1^\dagger\|_F^{2q}\big)\big]\Big)^\frac{1}{q}. \label{eq:inequality11}
\end{align}
Jensen’s inequality (using $q\leq \ell/2$) and Lemma \ref{Lq_psudo} give
\begin{align*}
    \mathbb{E}_{\Psi_1}\big(\|\Psi_1^\dagger\|_F^{2q}\big)&=\big[\mathbb{E}^q_{\Psi_1}\big(\|\Psi_1^\dagger\|_F^{2})\big]^q \leq\big[\mathbb{E}^{\frac{\ell}{2}}_{\Psi_1}\big(\|\Psi_1^\dagger\|_F^{2}\big)\big]^q =\big[\mathbb{E}^{\frac{\ell}{2}}_{\Psi_1}\big(\|(\Psi_1^T)^\dagger\|_F^{2}\big)\big]^q\\&\leq (r+p)^q\left(\frac{\Gamma(\frac{1}{2})}{2^{\ell/2}\Gamma(\ell/2+\frac{1}{2})}\right)^{\frac{2q}{\ell}}.
\end{align*} Inserting this inequality into~\eqref{eq:inequality11} yields
\begin{align*}
   \mathbb{E}^q\big(\|\Psi_1^\dagger \Psi_2 Q^T_\perp A(t)\|^2_F\big)&\leq \left(\frac{2^q\Gamma(\frac{1}{2}+q)}{\Gamma(\frac{1}{2})}\right)^\frac{1}{q}\left(\frac{\Gamma(\frac{1}{2})}{2^{\frac{\ell}{2}}\Gamma(\frac{\ell+1}{2})}\right)^{\frac{2}{\ell}}(r+p)\Big(\mathbb{E}_\Omega\big[\| Q^T_\perp A(t)\|^{2q}_F\big]\Big)^\frac{1}{q}
    \\&\leq \left(\frac{2^\frac{\ell}{2}\Gamma(\frac{1}{2}+\frac{\ell}{2})}{\Gamma(\frac{1}{2})}\right)^\frac{2}{\ell}\left(\frac{\Gamma(\frac{1}{2})}{2^{\frac{\ell}{2}}\Gamma(\frac{\ell+1}{2})}\right)^\frac{2}{\ell}(r+p)\Big(\mathbb{E}_\Omega\big[\| Q^T_\perp A(t)\|^{2q}_F\big]\Big)^\frac{1}{q}\\
    &\leq (r+p)\Big(\mathbb{E}_\Omega\big[\| Q^T_\perp A(t)\|^{2q}_F\big]\Big)^\frac{1}{q},
\end{align*}
 where the second inequality follows from \begin{equation*}
     {\left(\frac{2^q\Gamma(\frac{1}{2}+q)}{\Gamma(\frac{1}{2})}\right)^{\frac{1}{q}}}=\mathbb{E}^q(X)\leq \mathbb{E}^\frac{\ell}{2}(X)= {\left(\frac{2^\frac{\ell}{2}\Gamma(\frac{1}{2}+\frac{\ell}{2})}{\Gamma(\frac{1}{2})}\right)^{\frac{2}{\ell}}},
 \end{equation*} letting $X$ denote a $\chi^2$ random variable with 1 degree of freedom.
Combining with \eqref{eq:bound_ny} and using, once more, Jensen’s inequality shows
\begin{align*}
    \mathbb{E}^q(\|(I-\mathcal{P}_{A(t)\Omega,\Psi})A(t)\|_F^2)&\leq (1+r+p) \mathbb{E}^q\big(\|(I-QQ^T)A(t)\|^2_F\big)\\
     &\leq (1+r+p) \mathbb{E}^{\frac{p}{2}}\big(\|(I-\mathcal{P}_{A(t)\Omega})A(t)\|^2_F\big).
\end{align*}
In summary, we obtain \begin{align*}
     \mathbb{E}^q\Big(\int_{D}\|(I-\mathcal{P}_{A(t)\Omega,\Psi})A(t)\|_F^2\,\mathrm{d}t\Big)&\leq \int_{D}\mathbb{E}^q(\|(I-\mathcal{P}_{A(t)\Omega,\Psi})A(t)\|_F^2)\,\mathrm{d}t\\&\leq (1+r+p) \int_{D}\mathbb{E}^{\frac{p}{2}}\big(\|(I-\mathcal{P}_{A(t)\Omega})A(t)\|^2_F\big)\,\mathrm{d}t\\&\leq(1+r+p)(1+r)\int_{D}\sum_{j> r}\sigma_j\big(A(t)\big)^2\,\mathrm{d}t,
\end{align*} where the last inequality follows from~\eqref{eq:lp-inequaility}. The proof is concluded by applying the Markov inequality as in the proof of Theorem~\ref{l2-tail}.
\end{proof}

%Theorem \ref{Ny-l2-tail} shows that the $L^2$ approximation error is deviates, with high probability, by not much more than a factor $\sqrt{(1+r+p)(1+r)}$ from the best rank-$r$ approximation error.
Comparing the results of Theorems~\ref{l2-tail} and~\ref{Ny-l2-tail}, we see that there is an extra factor $\sqrt{1+r+p}$ in the quasi-optimality factor of Theorem~\ref{Ny-l2-tail} and the failure probability also depends on the oversampling parameter $\ell$.

\section{Numerical experiments}
\label{sect: numerical}
In this section, we verify the performance of randomized low-rank approximation methods for parameter-dependent matrices numerically.
If $\hat A(t)$ denotes the approximation to $A(t)$ returned by a method, we estimate the $L^2$ error on an interval $[\alpha,\beta]$,
\begin{equation} \label{eq:l2error}
 \Big( \int_\alpha^\beta \|\hat A(t) - A(t)\|_F^2\,\mathrm{d}t \Big)^{1/2}.
\end{equation}
Unless mentioned otherwise, we approximate the integral with the composite trapezoidal rule using 300 uniform quadrature points.
Since all methods involve random quantities, we report the mean $L^2$ approximation error as well as the spread between the largest and smallest errors for 20 independent random trials (indicated by lower/upper horizontal lines in the graphs).
%Also, we include the comparisons between the use of constant DRMs and independent DRMs for each parameter.

All experiments have been performed in Matlab (version 2023a) on a Macbook Pro with an Apple M1 Pro processor. The code used to produce the figures can be found at \url{https://github.com/hysanlam/Parameter\_RLR}. In the following, parameter-dependent HMT refers to our Matlab implementation of the approximation~\eqref{eq:parahmt}, using the QR decomposition of $A(t) \Omega$ for realizing the orthogonal projection $\mathcal P_{A(t)\Omega}$.
Parameter-dependent Nystr\"om refers to our Matlab implementation of Algorithm~\ref{Algro: GN_para}, using 
$\ell=0.2(r+p)$ and $\epsilon=2.22\times 10^{-15}$; see also~\cite{nakatsukasa2020fast}. 
\subsection{Synthetic example}
We first consider a synthetic example from~\cite[section 6.1]{ceruti2022unconventional}:
\begin{equation} \label{eq:syntheticexample}
    A(t)=e^{tW_1} e^tD e^{tW_2},\quad t\in [0,1],
\end{equation}
where $D\in\mathbb{R}^{n\times n}$ is diagonal with entries $d_{jj}=2^{-j}$ and $W_1\in\mathbb{R}^{n\times n}$, $W_2\in\mathbb{R}^{n\times n}$ are randomly generated, skew-symmetric matrices. {We compute the matrix exponential using Matlab's expm.} The singular values of $A(t)$ are $e^t2^{-j}$, $j = 1,\ldots,n$. Figure~\ref{fig:Synthetic} displays the results obtained from applying our methods to $A(t)$ for $n=100$. It turns out that the accuracy of our methods is within $1$--$2$ orders of magnitude relative to the optimal method, truncating the SVD to rank $r+p$ for each $A(t)$. 
 \begin{figure}[ht]
\includegraphics[width=\textwidth]{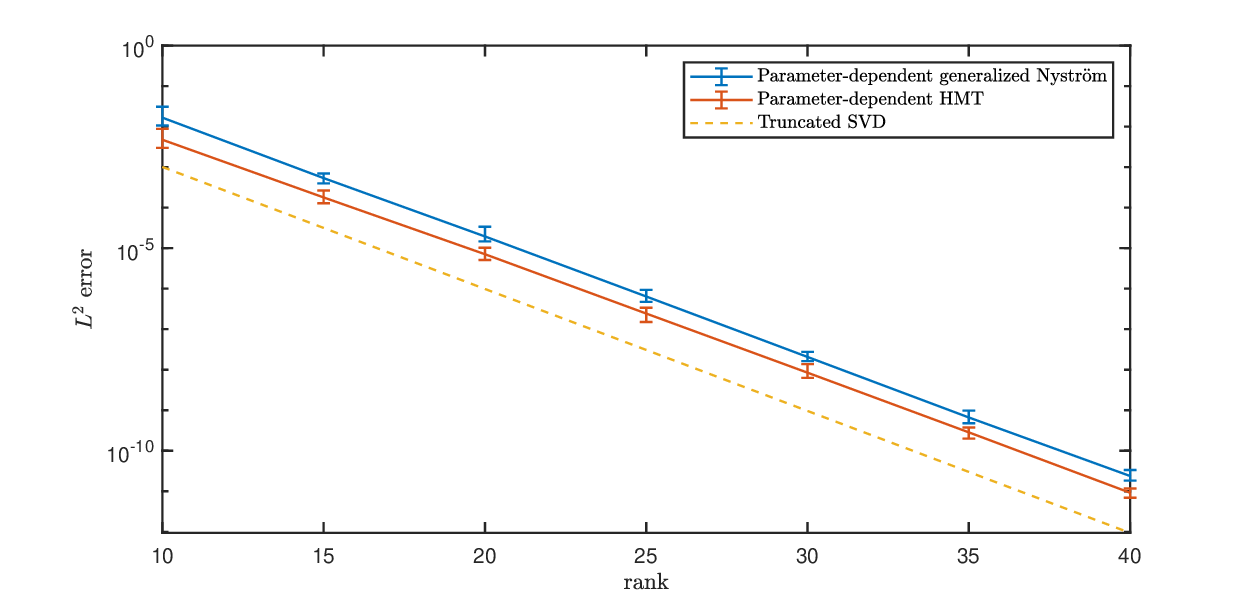}
\centering
\caption{Synthetic example~\eqref{eq:syntheticexample} with $n = 100$. $L^2$ approximation error of parameter-dependent HMT / Nystr\"om compared to the pointwise truncated SVD for different approximation ranks.}
\label{fig:Synthetic}
\end{figure}

Figure~\ref{fig:Synthetic_constant}
also compares the error of our methods with the one obtained from applying HMT / Nystr\"om with different, independent DRMs for each parameter value $t$ used in the approximation of the integral in~\eqref{eq:l2error}. {We have additionally plotted boxes; which the lower and upper end of each box denotes the 25th and 75th percentiles, respectively.} Our methods exhibit a somewhat larger variance in the error but still give similar accuracy. Similar observations have been made for all other examples considered in this section; using constant instead of independent Gaussian random DRMs only has a minor impact on accuracy. We will therefore not report them in detail for the other examples. 
   
 \begin{figure}[H]
\includegraphics[width=\textwidth]{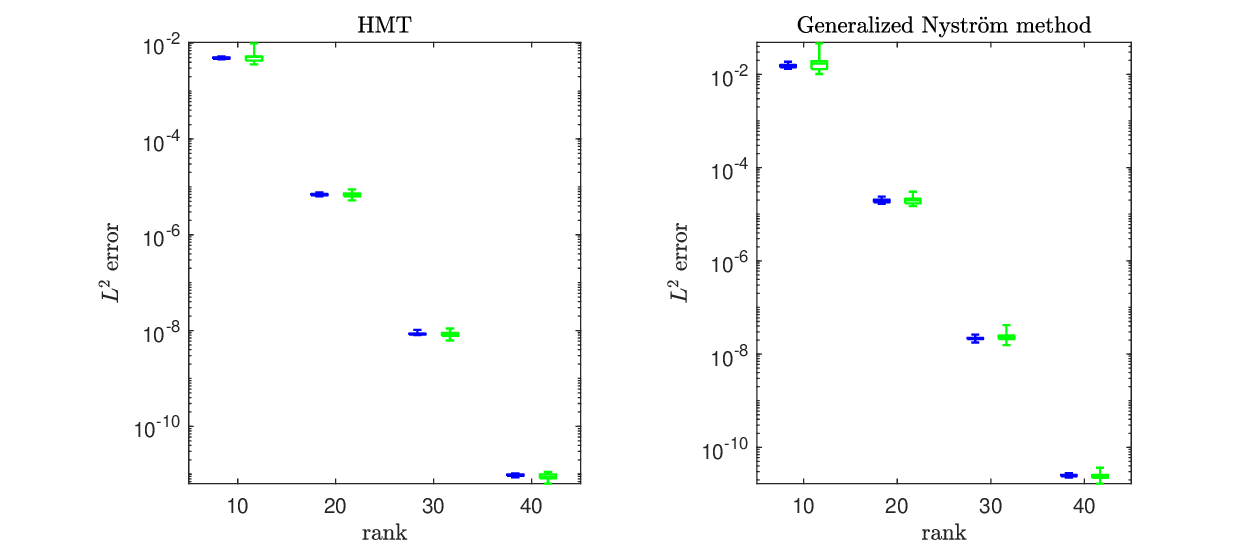}
  \caption{Synthetic example~\eqref{eq:syntheticexample} with $n = 100$. $L^2$ approximation error of parameter-dependent HMT / Nystr\"om (green) compared to pointwise HMT / Nystr\"om with different, independent DRMs (blue) for approximation ranks $10$, $20$, $30$, and $40$.}
\label{fig:Synthetic_constant}
\end{figure}

\subsection{Parametric cookie problem}

We perform low-rank approximation to the solution of a time- and parameter-dependent discretized PDE from~\cite{carrel2022low,kressner2011low}, sometimes called the parametric cookie problem. In matrix form, the problem was given by\begin{equation} \label{eq:paracookie}
    \dot{A}(t) =-B_0A(t) -B_1A(t) C+b\mathbf{1}^T,\;\;\;A(t_0)=A_0,
\end{equation} where $B_0,B_1\in\mathbb{R}^{1580\times1580}$ are the mass and stiffness matrices, respectively, $b\in\mathbb{R}^{1580}$ is the discretized  inhomogeneity, and $C=\diag(0,1,2,\ldots,100)$ contains the parameter samples on the diagonal; see~\cite{kressner2011low}. {We set $t_0=-0.01$ and $A_0$ as the zero matrix. After computing the exact solution of \eqref{eq:paracookie} at $t=0$, we estimate the low-rank approximation methods error for $A(t)$ on $[0,0.9]$ numerically. We first discretize the interval $[0,0.9]$ with 300 uniform points, then we obtain an evaluation of $A$ on these points by using Matlab's ode45 with tolerance parameters $\{\text{'RelTol'}, 1e-12, \text{'AbsTol'}, 1e-12\}$ on each subinterval. Finally, we apply our methods and approximate \eqref{eq:l2error} by trapezoidal rule on every subinterval.} Figure~\ref{fig:cookie_sig} confirms that $A(t)$ has rapidly decaying singular values for different $t$ and thus admits good low-rank approximations. Figure~\ref{fig:cookie_error} displays the resulting errors. Again, our methods perform only insignificantly worse than the truncated SVD.
 \begin{figure}[H]
\includegraphics[width=\textwidth]{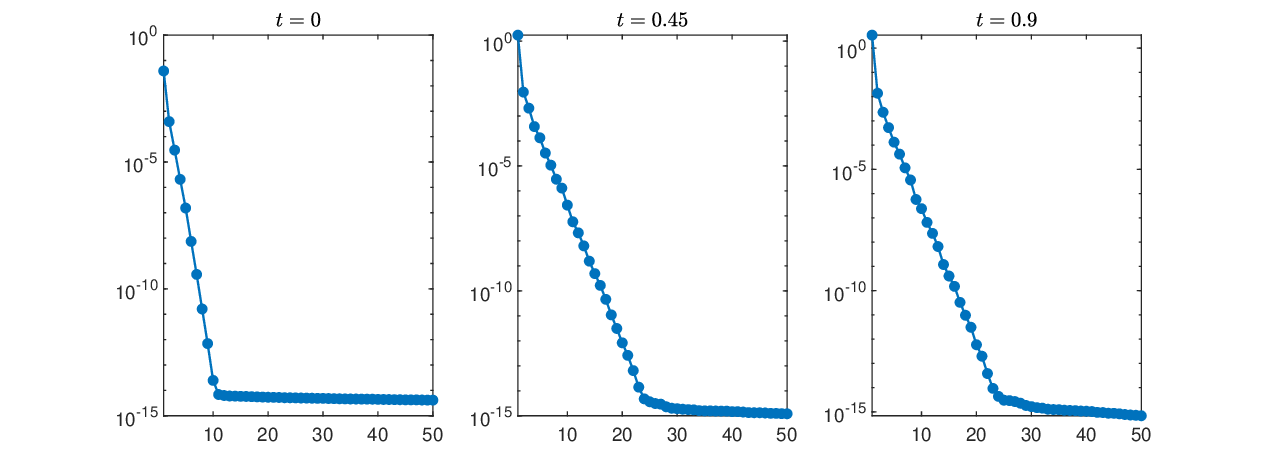}
\centering
\caption{Singular values of the reference solution to~\eqref{eq:paracookie} at $t=0,0.45$ and $0.9$.}
\label{fig:cookie_sig}
\end{figure}

 \begin{figure}[H]
{\includegraphics[width=\textwidth]{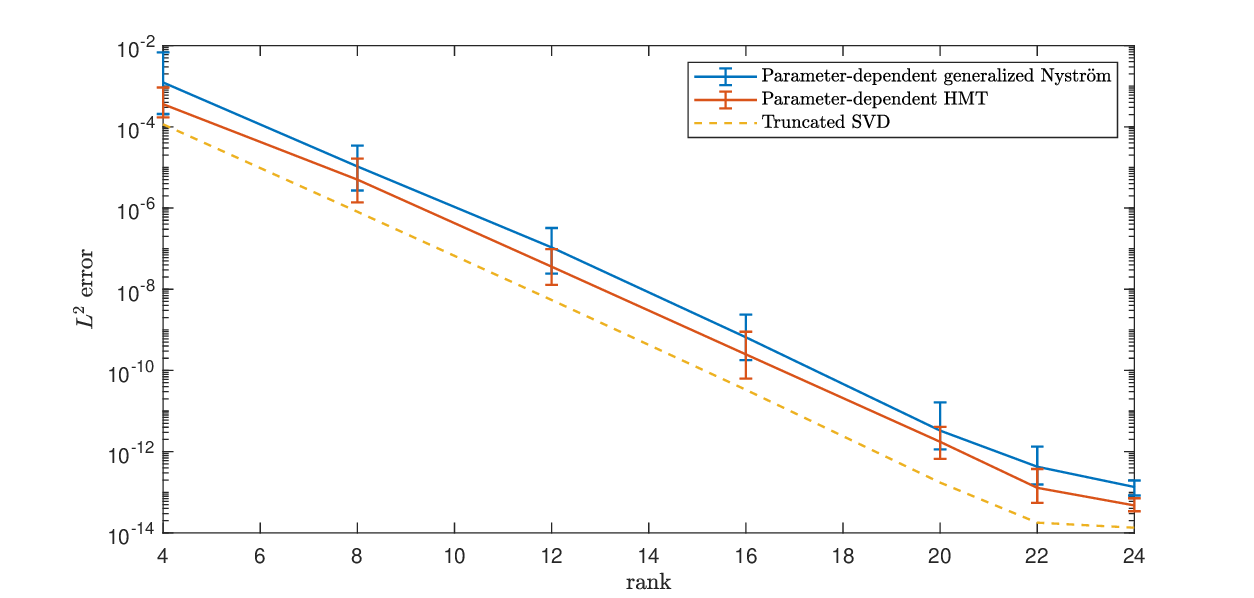}}
\centering
\caption{Parametric cookie example~\eqref{eq:paracookie}. $L^2$ approximation error of parameter-dependent HMT / Nystr\"om compared to the pointwise truncated SVD for different approximation ranks.}
\label{fig:cookie_error}
\end{figure}
% 
% We also compare the error between the randomized low-rank approximation using constant DRMs and independent DRMs for each parameter by showing the box plot of the $L^2$ of 20 trials in figure \ref{fig:cookie_constant}. Again, although the errors of using the constant DRMs have a larger spread, it gives comparable errors to using independent DRMs for each parameter.
% \begin{figure}[h]
%     \includegraphics[scale=0.18]{cookie_error_compare_with_constant_HMT.jpg}
% \hfil
%     \includegraphics[scale=0.18]{cookie_error_compare_with_constant_GN.jpg}
%     \caption{The box plots for $L^2$ approximation error of 20 random trials of randomized low-rank approximation methods with independent DRMs (Blue) and with constant DRMs (Green).}
% \label{fig:cookie_constant}
% \end{figure}

\subsection{Discrete Schr\"{o}dinger equation in imaginary time}

In this example, we perform low-rank approximation to the solution of the discrete Schr\"{o}dinger equation in
imaginary time from~\cite{ceruti2022unconventional}:
\begin{equation}\label{eq:discschroedinger}
    \dot{A}(t) =-H[A(t)],\quad A(0)=A_0,
\end{equation} where
\begin{equation*}
    H[A(t)]=-\frac{1}{2}(DA(t)+A(t) D)+V_{\cos} A(t) V_{\cos}\in \mathbb{R}^{n\times n},
\end{equation*}
$D=\text{tridiag}(-1,2,-1)$ is the discrete 1D Laplace, and $V_{\cos}$ is the diagonal matrix with diagonal entries
$1-\cos(2j\pi/n)$ for $j=-n/2,\ldots,n/2-1$.
We choose $n=2048$ and an initial value $A_0$ that is randomly generated with prescribed singular values $10^{-i}$, $i=1,\ldots,2048$.
%The function $H$ in the differential equation arises from Hamiltonian $\frac{-1}{2}\Delta_{discrete}+V(x)$ on a equidistant grid, where $V(x)$ is the torsional potential $(1-\cos({x_1}))(1-\cos({x_2}))$.
Similar to the previous example, to estimate the low-rank approximation error to $A(t)$ on $[0,0.1]$ numerically, we first discretize the interval with 300 uniform points. Then, obtain the evaluation of $A(t)$ on these points by using Matlab's ode45 with tolerance parameters $\{\text{'RelTol'}, 1e-12, \text{'AbsTol'}, 1e-12\}$ on each subinterval and approximate \eqref{eq:l2error} by trapezoidal rule on every subinterval. Comparing Figures~\ref{fig:cookie_sig} and~\ref{fig:schrodinger_sig}, we see that the singular value decay is now somewhat less favorable. In turn, a higher rank is needed to attain the same approximation error. Figure~\ref{fig:schrodinger_error} once again confirms the good performance of our methods relative to the truncated SVD. 
\begin{figure}[ht]
\includegraphics[width=\textwidth]{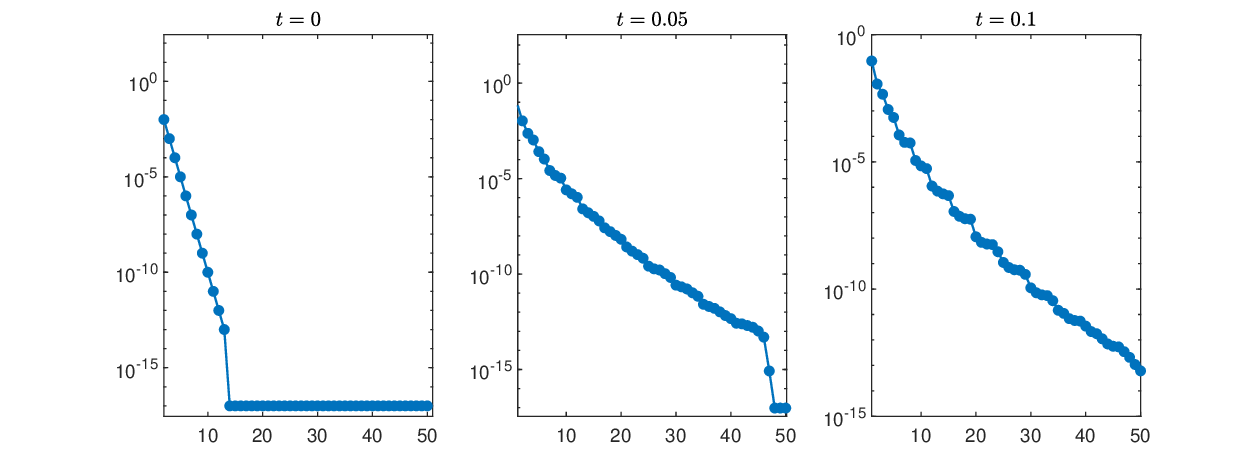}
\centering
\caption{The singular values of the reference solution to the discrete Schr\"{o}dinger equation at $t=0,0.05$ and $0.1$.}
\label{fig:schrodinger_sig}
\end{figure}%
\begin{figure}[ht]
\includegraphics[width=\textwidth]{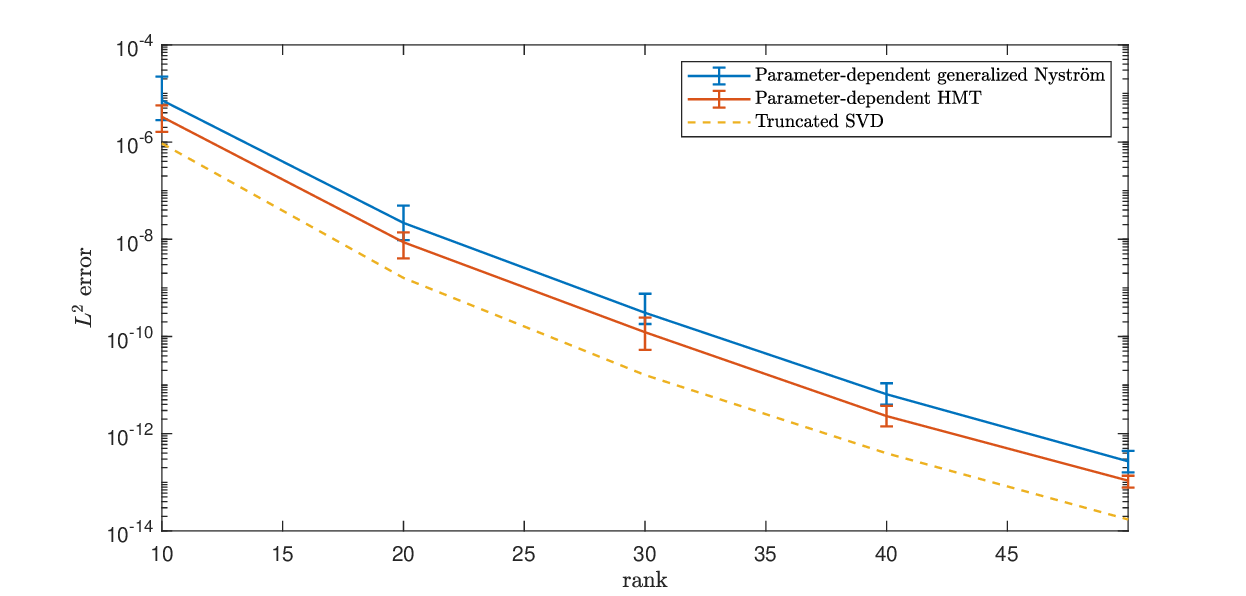}
\centering
\caption{Discrete Schr\"{o}dinger example~\eqref{eq:discschroedinger}. $L^2$ approximation error of parameter-dependent HMT / Nystr\"om compared to the pointwise truncated SVD for different approximation ranks.}
\label{fig:schrodinger_error}
\end{figure}

% In figure \ref{fig:Schrodinger_constant}, we observe that in the 20 random trials, in most cases, using the constant DRMs or using the independent DRMs gives similar errors.
% \begin{figure}[H]
%     \includegraphics[scale=0.18]{Schrodinger_error_compare_constant_HMT.jpg}
% \hfil
%     \includegraphics[scale=0.18]{Schrodinger_error_compare_constant_GN.jpg}
%     \caption{The box plots for $L^2$ error of 20 random trials of randomized low-rank approximation methods with independent DRMs (Blue) and with constant DRMs (Green).}
% \label{fig:Schrodinger_constant}
% \end{figure}

\subsection{Parameterized Gaussian covariance kernel}
\subsubsection*{Test 1. Data from regular grid}
We first consider an example from~\cite[section 5.1]{kressner2020certified} concerned with the low-rank approximation of Gaussian covariance matrices.
Consider the Gaussian covariance kernel on the domain $\mathbb D = [0,1]\times [0,1]$ given by
\begin{equation*}
    c(x,y;t)=\exp\left(-\frac{\|x-y\|_2^2}{2t^2}\right),\quad \forall x,y \in \mathbb D,
\end{equation*}
where $t\in \Theta:=[0.1,\sqrt{2}]$ is a parameterized correlation length. We proceed as in~\cite[section 5.1]{kressner2020certified}, we discretize the domain $\mathbb D$ by a regular grid with $n=4900$ grid points and denote each grid point as $\mathbf{x}_i\in \R^2$ for $i=1,\ldots,n$. This leading to the $n\times n$ covariance matrix $C(t)$ with entries
\begin{equation*}
    C(t)_{i,j}:=\frac{1}{n}c(\mathbf{x}_i,\mathbf{x}_j,t),\quad i,j=1,\ldots,n.
\end{equation*}
We now approximate the kernel by an expansion that separates the spatial variables from the parameter $t$:
\begin{equation*}
    c(x,y;t)\approx \sum^s_{j=1}\phi_j(t)a_j(\|x-y\|_2),
\end{equation*} where $\phi_j:[0.1,\sqrt{2}]\rightarrow \mathbb{R}$ and $a_j:\mathbb{R}\rightarrow \mathbb{R}$.
As described in~\cite{kressner2020certified}, we use a polynomial expansion computed by the \emph{Chebfun 2} command \textbf{cdr}, see \cite{Townsend2014computing}.
In figure \ref{fig:l_infty}, we show the error associated with the
approximate covariance kernel for different values of $s$.
 \begin{figure}[h]
\includegraphics[width=\textwidth]{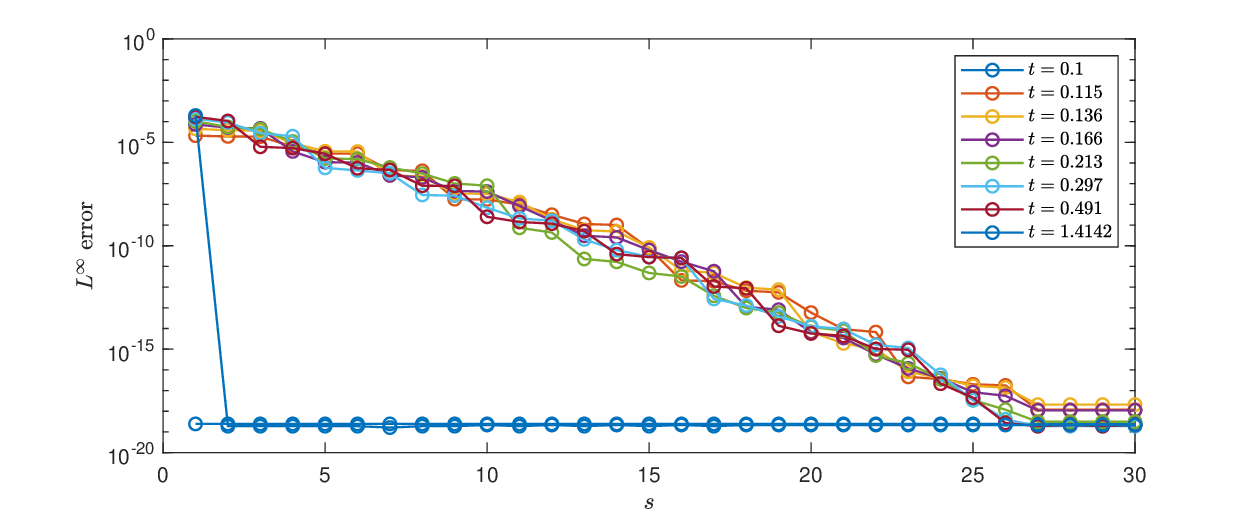}
\centering
\caption{Error associated with the
approximate covariance kernel $\sum^s_{j=1}\phi_j(t)a_j(\|x-y\|_2)$ vs. expansion length $s$. Each curve represents the $L^\infty$ norm of $c(x,y;t)-\sum^s_{j=1}\phi_j(t)a_j(\|x-y\|_2)$ on $\mathbb D$ for fixed $t$.}
\label{fig:l_infty}
\end{figure}
In the following, we choose $s=18$, leading to a negligible error. By defining $A_k$ to be the matrix associated with $a_k(\|\mathbf{x}_i-\mathbf{x}_j\|_2)$, we obtain the affine linear decomposition
\begin{equation} \label{eq:Atheta}
    C(t)\approx A(t):=\frac{1}{n}\sum^s_{k=1}\phi_k(t)A_k,
\end{equation}
which allows us to use Algorithms~\ref{offline phase 2} and~\ref{online phase 2} for HMT and Algorithms~\ref{GN-offline phase 2} and~\ref{GN-online phase 2}
for generalized Nystr\"{o}m for approximating $C(t)$ via $A(t)$. 
 {The parameter-dependent ACA (adaptive cross approximation) from~\cite{kressner2020certified} is also applicable in this example. The parameter-dependent ACA does not require the whole matrix to be computed, and to obtain the low-rank factors, it has a cost that scales linearly with respect to $n$, see~\cite[section 3.5]{kressner2020certified}. The parameter-dependent ACA returns a subset of indices $I$ by controlling the diagonal part of the ``updated" $A(t)$ on a finite surrogate set $\Theta_A\subset \Theta$. With the indices $I$, the error of the low-rank approximation to $C(t)$ by $C(t)(:I)[C(t)(I,I)]^{-1}C(t)(:I)^T$ are regulated for all $t\in \Theta_A$.} See \cite[Section 3]{kressner2020certified} for details. In this example, we choose $\Theta_A$ to be 300 uniform points in $\Theta$.
 
{We first compare the computational time of our methods and parameter-dependent ACA for different approximation ranks. We count the computational time needed for computing the low-rank approximation of $C(t)$ at every $t\in \Theta_A$ in factored form. In this comparison, we consider the setting that $A(t)$ is explicitly available for all the methods. To obtain a more accurate computational time, we repeat the whole computation 10 times (with the same DRMs realization) and report the average time in Table \ref{table:aca_time}. We observe that when the rank is not high, ACA performs significantly faster than both of our methods. However, when the rank increases, the computational time of the ACA method increases more significantly. This behaviour can be explained by the complexity bound, by \cite{kressner2020certified}, the parameter-dependent ACA costs $O(nr^2s^2+|\Theta _A|(r^4s+r^3s^2))$ to obtain the indices set $I$.}
\begin{table}[h]
\centering
\caption{Parameterized covariance matrix~\eqref{eq:Atheta} with $\mathbf{x_j} \in \mathbb{D}$ and $t\in \Theta= [0.1,\sqrt{2}]$. Computational time (in seconds).}
\begin{tabular}{||c | c c c ||} 
 \hline
 \multicolumn{4}{||c||}{Average computational time in seconds} \\
  \hline
\scriptsize{rank}&  \scriptsize{parameter-dependent HMT} & \scriptsize{parameter-dependent generalized Nystr\"{o}m} & \scriptsize{parameter-dependent ACA} \\[0.5ex] 
 \hline\hline
  10 & 5.3258 & 3.1637&0.8305 \\ 
 \hline
 20 & 7.8558 &3.7768 &1.7703\\
 \hline
 30 & 10.4207 &4.3503 & 3.5445\\
 \hline
  40 & 13.1245 & 4.8606 & 6.3856 \\
 \hline
50 & 16.3159 & 5.7355& 10.6710 \\
 \hline
 60 & 20.6456 & 6.3244& 17.2224 \\
 \hline
\end{tabular}
\label{table:aca_time}
\end{table}

To compare the accuracy, we estimate the $L^2$ error with respect to approximate $C(t)$ by composite trapezoidal rule on $\Theta_A$ and report it in Figure~\ref{fig: A theta error}.
The error of all methods is within $1$--$2$ orders of magnitude of the best approximation error. HMT and ACA perform comparably well, while the error returned by generalized Nystr\"{o}m is up to a factor 10 higher.

 \begin{figure}[H]
\includegraphics[width=\textwidth]{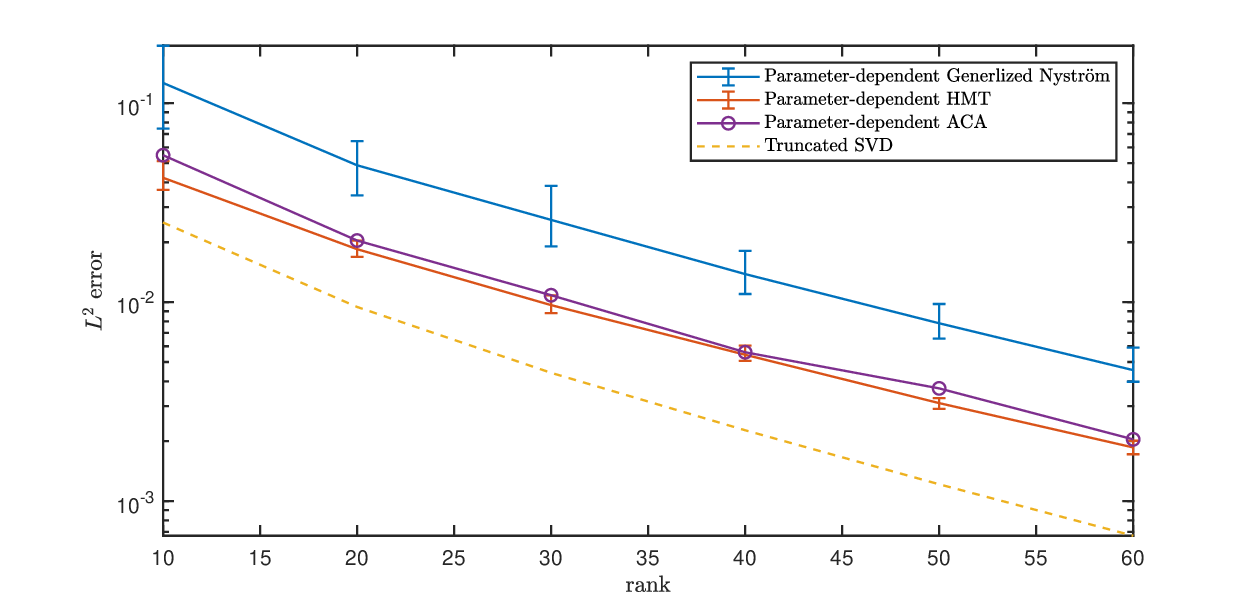}
\centering
\caption{Parameterized covariance matrix~\eqref{eq:Atheta} with $\mathbf{x_j} \in \mathbb{D}$ and $t\in \Theta= [0.1,\sqrt{2}]$. $L^2$ approximation error of parameter-dependent HMT / Nystr\"om / ACA compared to the pointwise truncated SVD for different approximation ranks.}
\label{fig: A theta error}
\end{figure}
\subsubsection*{Test 2. High-dimensional data}
To test the robustness of the method for high-dimensional data, motivated by \cite[Test 4]{Cai2023Data}, we follow the same setting as in Test 1 except this time we construct $C(t)$ and $A(t)$ with $\mathbf{x_j} \in \mathbb{R}^{128}$ for $j=1,\ldots, 4900$ and $t\in \Theta= [10,120]$. The data $\mathbf{x_j} $ were randomly chosen with replacement from the Gas Sensor Array Drift dataset obtained from the UCI Machine Learning repository \cite{asuncion2007uci}, which was standardized to have a mean of zero and a variance of one along each dimension. We report the error in Figure \ref{fig: A theta error_test2}. We observe that our method still gives accurate results while the parameter-dependent ACA deviates from the optimal approximation error the most. The non-robustness of ACA to high-dimensional data has also been observed in \cite[Test 4]{Cai2023Data}.
 \begin{figure}[h]
\includegraphics[width=\textwidth]{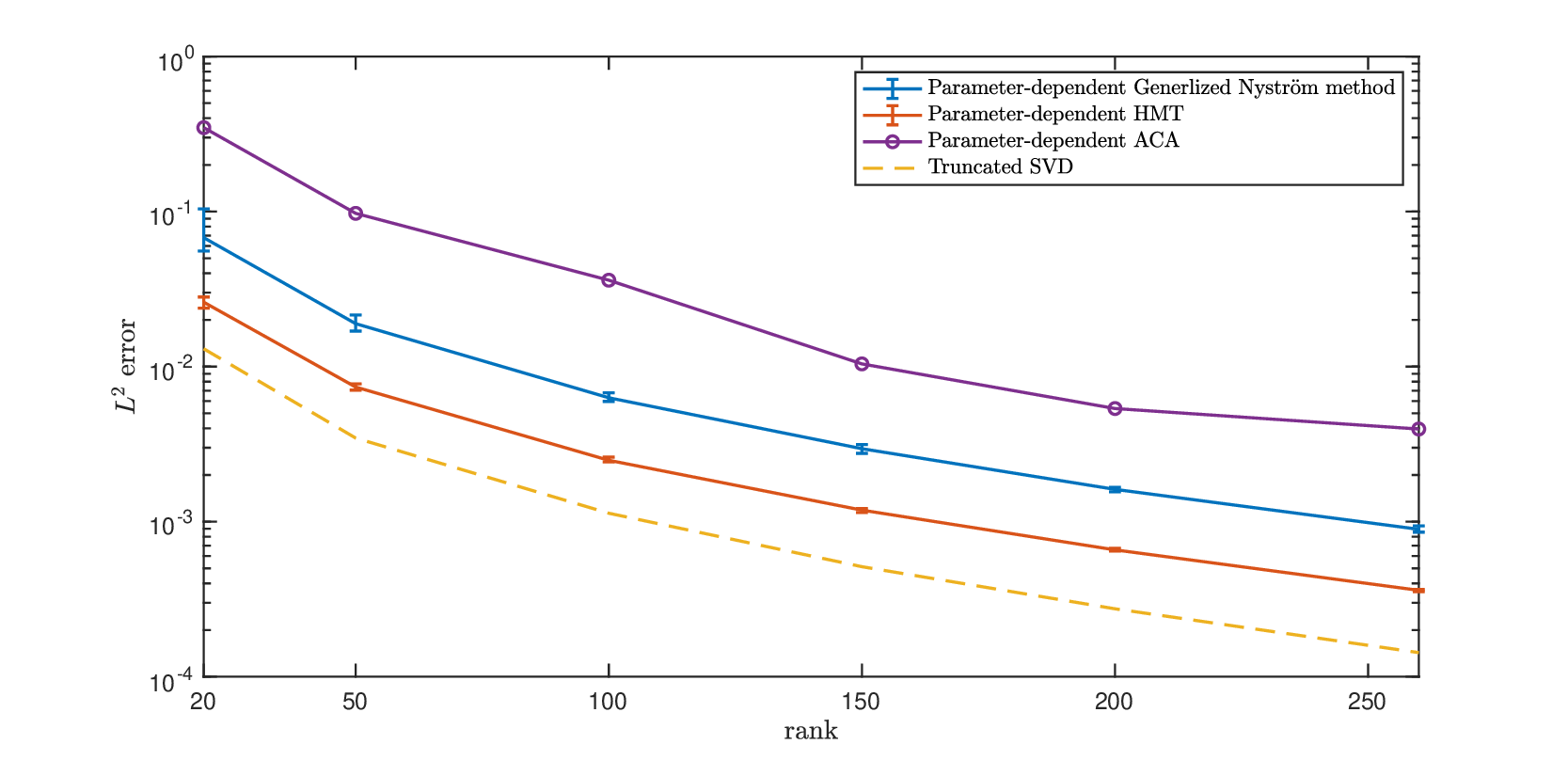}
\centering
\caption{Parameterized covariance matrix~\eqref{eq:Atheta} with $\mathbf{x_j} \in \mathbb{R}^{128}$ randomly sampled from Gas Sensor Array Drift dataset  and $t\in \Theta= [10,120]$. $L^2$ approximation error of parameter-dependent HMT / Nystr\"om / ACA compared to the pointwise truncated SVD for different approximation ranks.}
\label{fig: A theta error_test2}
\end{figure}
\subsubsection*{Test 3. Data with cluster structure}
The randomized nature of our proposed method makes it less sensitive to the geometric structure of the data compared to the parameter-dependent ACA, we illustrate it with the following example, which closely follows \cite[Test 5]{Cai2023Data}. We first take the same setting as in Test 1, but we randomly generate 900 vectors $\mathbf{x_j} \in \mathbb{R}^{2}$, split into three clusters (from left to right) with 100, 700, and 100 points respectively. See Figure \ref{fig: A theta error_test3_points}. Also, we set $t\in \Theta=[0.2,0.4]$.  The cluster structure of the data is not favorable for ACA; it may fail to enhance the accuracy of the approximation regardless of the increase in rank, see \cite[Test 5]{Cai2023Data} for more details.  We report the error in Figure \ref{fig: A theta error_test3}. We observe that the parameter-dependent ACA stagnate from $r=20$ to $r=28$. Similar stagnate behaviour is also observed in \cite[Test 5]{Cai2023Data} for the case of a fixed $\theta$. On the other hand, the parameter-dependent HMT and Nystr\"om still perform well and achieve sub-optimal results.
 \begin{figure}[H]
\includegraphics[width=\textwidth]{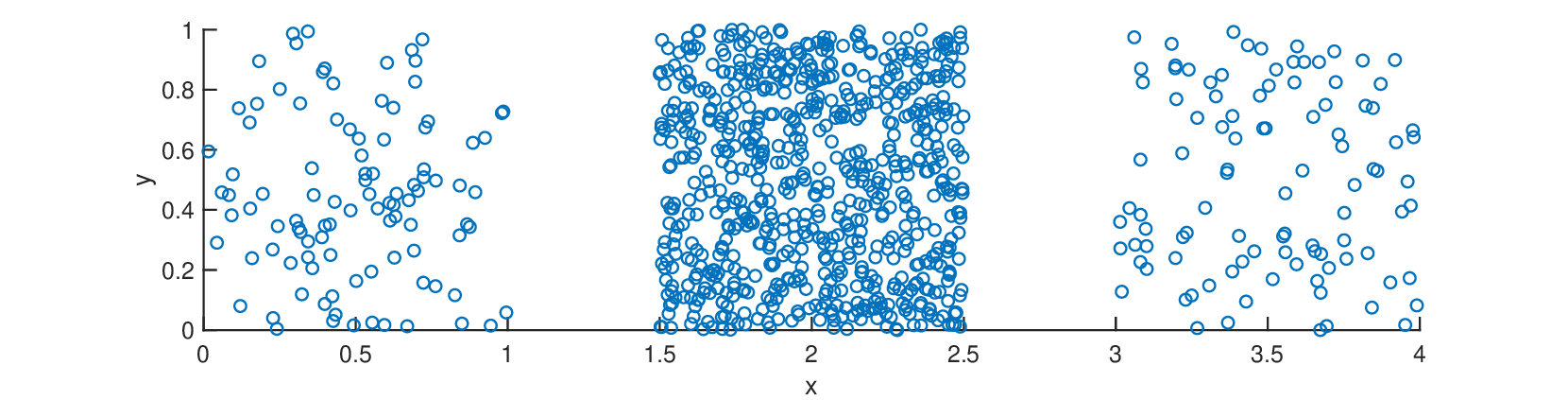}
\centering
\caption{Data $\mathbf{x} \in \mathbb{R}^{2}$, split into three clusters (left to right) with 100, 700, and 100 points respectively.}
\label{fig: A theta error_test3_points}
\end{figure}
 \begin{figure}[H]
\includegraphics[width=\textwidth]{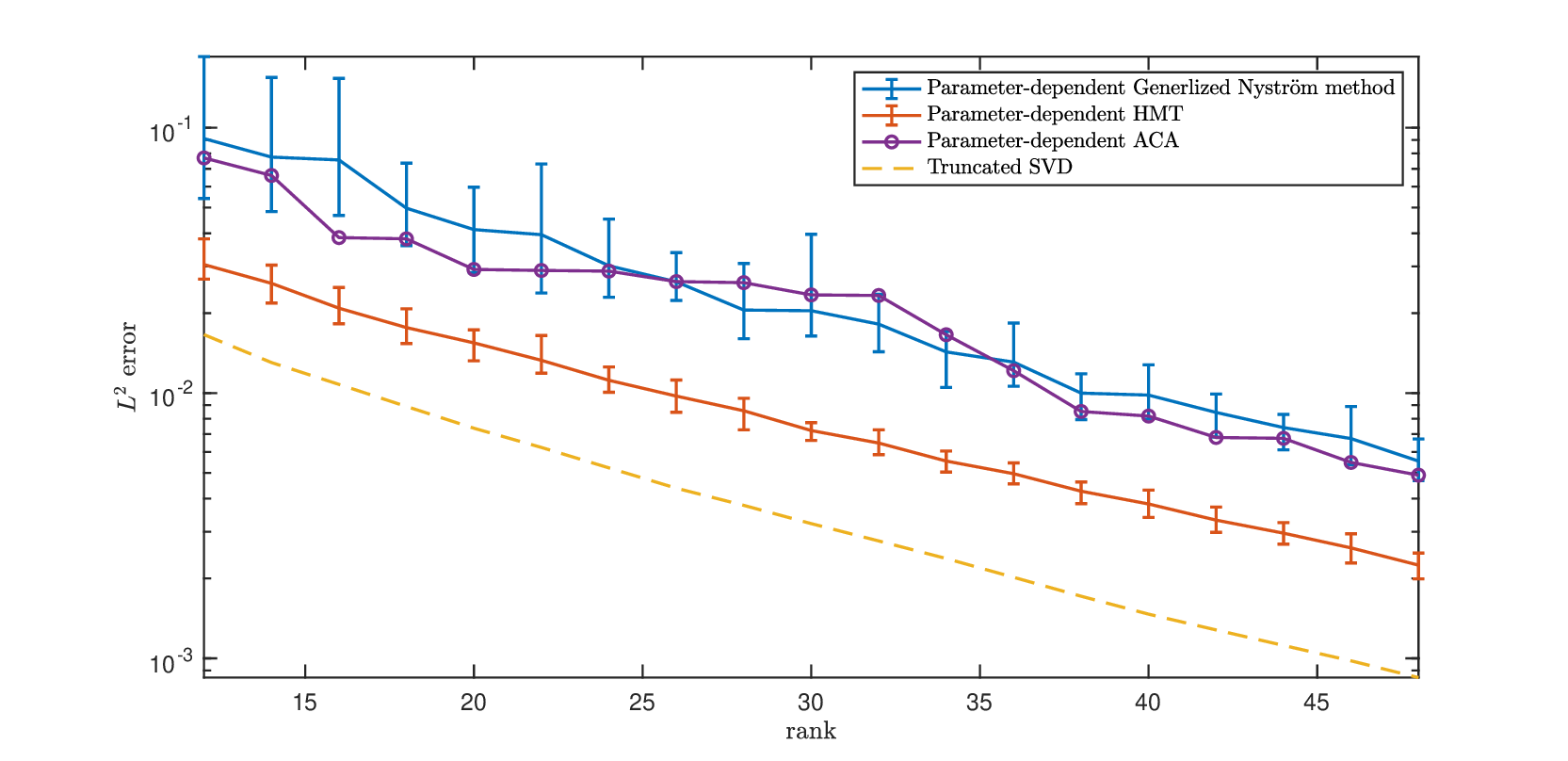}
\centering
\caption{Parameterized covariance matrix~\eqref{eq:Atheta} with clustered data $\mathbf{x_j} \in \mathbb{R}^{2}$ and  $t\in \Theta=[0.2,0.4]$. $L^2$ approximation error of parameter-dependent HMT / Nystr\"om / ACA compared to the pointwise truncated SVD for different approximation ranks.}
\label{fig: A theta error_test3}
\end{figure}

\section{Conclusion}

In this work, we extended two popular randomized low-rank approximation methods to a parameter-dependent matrix. The use of constant DRMs in our methods avoids the cost of generating DRMs and also reduces the computational cost of constructing the sketches. The theoretical and numerical results of this work show that our methods can be expected to achieve comparable accuracy compared to the truncated SVD and using independent DRMs for each parameter value.

\begin{paragraph}{Acknowledgments.}
 The authors sincerely thank Gianluca Ceruti for helpful discussions related to numerical experiments presented in this work and Stefano Massei for providing the implementation of parameter-dependent adaptive cross approximation. We also thank the referees for their constructive remarks.
\end{paragraph}

\bibliographystyle{siam}
\bibliography{main}
\end{document}